\documentclass[11pt]{amsart}        % Your input file must contain these two
\usepackage{Mytheorems}
\usepackage{amssymb,amscd,amsthm,amsxtra}
\usepackage{graphicx}
\usepackage{epstopdf}
\usepackage{latexsym}
\usepackage{mathtools}%commonly used math package
\usepackage{amsfonts}%commonly used math package for font typefaces
\usepackage{mathrsfs}%commonly used math package for \mathscr
\usepackage{amsthm}%commonly used math package for theorems
\usepackage{xfrac}%provide various types of fractions
\usepackage{color}
\usepackage{yhmath}%big \widetilde

\newtheorem{thm}{Theorem}[section]
\newtheorem{lem}[thm]{Lemma}
\newtheorem{rmk}[thm]{Remark}
\newtheorem{defi}[thm]{Definition}
\newtheorem{prop}[thm]{Proposition}
\newtheorem{cor}[thm]{Corollary}
\newtheorem{assump}[thm]{Assumption}
\newtheorem*{claim*}{Claim}
\usepackage{psfrag}
\usepackage{mathrsfs,amscd}
\usepackage[colorlinks,linkcolor=blue,anchorcolor=blue,citecolor=blue,bookmarks=true,pdfborder={0 0 0}]{hyperref}%pdfbookmark, equation link, ref link
\input epsf

\def\F{\mathcal{F}}%Fourier integral
\def\R{\mathbb{R}}%the set of real numbers
\def\S{\mathbb{S}}%the set of unit vectors
\def\O{\Omega}% the bounded convex domain
\DeclareMathOperator{\dist}{d}%distance function
\DeclareMathOperator{\diam}{diam}%diameter function
\def\SO{S_{\Omega}}%S_Omega operator

\begin{document}
\title{A revisit of the velocity averaging lemma: on the regularity of stationary Boltzmann equation in a bounded convex domain}
\author[IC]{I-Kun Chen}
\address{(I.-K. Chen)Institute of Applied Mathematical Sciences, National Taiwan University, No. 1, Sec. 4, Roosevelt Rd., Taipei 10617, Taiwan }
\email{ikun.chen@gmail.com}
\author{Ping-Han Chuang}
\address{(P.-H. Chuang)Department of Mathematics, National Taiwan University, No. 1, Sec. 4, Roosevelt Rd., Taipei 10617, Taiwan }
\email{d09221003@ntu.edu.tw}
\author{Chun-Hsiung Hsia}
\address{(C.-H. Hsia)Institute of Applied Mathematical Sciences, National Taiwan University, No. 1, Sec. 4, Roosevelt Rd., Taipei 10617, Taiwan }
\email{willhsia@math.ntu.edu.tw}
\author{Jhe-Kuan Su}
\address{(J.-K. Su)Department of Mathematics, National Taiwan University, No. 1, Sec. 4, Roosevelt Rd., Taipei 10617, Taiwan }
\email{hsnu127845@gmail.com}
\date{\today}
\maketitle
\begin{abstract}
    In the present work, we adopt the idea of velocity averaging lemma to establish regularity for stationary linearized Boltzmann equations 
    in a bounded convex domain. Considering the incoming data, with three iterations, we establish regularity in fractional Sobolev space in space 
    variable up to order $1^-$.  \vspace{0.5cm}\\
    \noindent $Keywords:$ Boltzmann equation; regularity; averaging lemmas; fractional Sobolev spaces %  \vspace{0.5cm}\\ %
%\noindent \textbf{All figures can be reproduced in print in black and white.}%
\end{abstract}

	\tableofcontents
    \section{Introduction}
    The celebrated velocity averaging lemma reveals that the combination of transport and averaging in velocity yields regularity in space variable \cite{golse85,golse88}. 
    This is one of the key features that DiPerna and Lions used to attack the Cauchy problem for Boltzmann equations \cite{diperna89}. 
    It is natural to adopt this technique to the study  of  regularity problem of linearized Boltzmann equation in the whole space \cite{chuang19}. 
    In \cite{golse88}, in addition to the 
    applications to the whole space domains, the authors also investigate the applications to bounded convex domains for transport equations. 
    By adopting zero extension, they reduce the bounded domain case to the whole space case. 
    In contrast with the whole space case in \cite{chuang19}, which the regularity can be improved indefinitely 
    by iterations, when applying the trick in \cite{golse88} in a bounded domain, one can only 
    proceed for one iteration. 
    Notice that the main tool of velocity averaging lemma, namely,  the method of Fourier transform, does not translate well on a bounded domain. In this article, we adopt Slobodeckij semi-norm as an alternative concept of Sobolev function class. This  shifts the difficulty to singular integrals. To calculate these integrals, we encounter some estimates related to the geometry of the boundary.  The way we do it is to properly compare convex domains with spherical domains in $\R^3$ 
    so that we can build up estimates for convex domains based on those for spherical domains. Another key feature to prove the boundedness of the singular integrals is the change of variables which will be addressed in Lemma~\ref{lemma:ChangeOfVariable1} and Lemma~\ref{lemma:ChangeOfVariable2}. 
    Considering the incoming data, with three iterations, we establish regularity in fractional Sobolev space in space 
    variable up to order $1^-$.

    Recall the velocity averaging lemma in \cite{desvillettes03,golse88}.  Suppose $u$ is an $L^2$ solution to the transport equation
    \begin{equation*}
    v\cdot\nabla_x u=G(x,v),\ \ \ \ \ (x,v)\in\mathbb{R}^n\times\mathbb{R}^n,
    \end{equation*}
    where $G\in L^2.$ Let
    \begin{equation*}
    \bar u(x):= \int_{\mathbb{R}^n} u(x,v)\psi(v)\,dv,
    \end{equation*}
    where $\psi$ is a bounded function with compact support. 
    Then, we have 
    \begin{equation*}
    \bar u(x) \in \tilde{H}^{\sfrac12}(\mathbb{R}^n).
    \end{equation*}
    Here, the Sobolev space is  generalized to non-integer order via the Fourier transform as follows.
    
    \begin{defi}\label{SobolveFourier}
    We say $u:\mathbb{R}^3\to \mathbb{R}$ is in $\tilde{H}^s_x(\R^3)$ if
    \begin{equation}
    \|u\|_{\tilde{H}^s_x(\R^3)}=\left(\int_{\R^3}(1+|\xi|^2)^{ s}|\F(u)(\xi)|^2\,d\xi \right)^\frac{1}{2}<\infty,
    \end{equation} 
    where $\F(u)(\xi)$ is the Fourier transform of $u$, i.e.,
    \begin{equation*}
    \F(u)(\xi)=(2\pi)^{-\frac32}\int_{\mathbb{R}^3}u(x)e^{-i\xi\cdot x}\, d x.
    \end{equation*}
    \end{defi}
    The velocity averaging lemma demonstrates that the regularity in the transport direction can  be converted  to the regularity in space variable after averaging with weight $\psi$.

    %%%%%%%%%%%%%%%%%%%%%%%%%%%%%%%%%%%%%%%%%%%%%%%%%%%%%%%%%%%%%%%%%%%%%%%%%%%%%%%%%%%%%%%%%%%%%%%%%%%%%%%%%%%%%%%%%%%%%%%%%%%%%%%%%%%%%%%%%%%%%%%%%%%%%%%%%%%%%%%%%%%%%%%
    %%%%%%%%%%%%%%%%%%%%%%%%%%%%%%%%%%%%%%%%%%%%%%%%%%%%%%%%%%%%%%%%%%%%%%%%%%%%%%%%%%%%%%%%%%%%%%%%%%%%%%%%%%%%%%%%%%%%%%%%%%%%%%%%%%%%%%%%%%%%%%%%%%%%%%%%%%%%%%%%%%%%%%%
    First, we recapitulate the stationary linearized Boltzmann equation in the whole space, 
    
    \begin{equation}
    v\cdot \nabla_xf(x,v)=L(f),
    \end{equation}
    where $f$ is the velocity distribution function and $L$ is the linearized collision operator. 
    The linearized collision operator under consideration can be decomposed into a multiplicative operator and an integral operator. 
    \begin{equation}\label{LinearizedBE}
    L(f)=-\nu(v)f+K(f),
    \end{equation}
    where 
    \begin{equation*}
    K(f)=\int_{\R^3} k(v,v_*)f(v_*)\,dv_*.
    \end{equation*}
    Therefore, we can rewrite \eqref{LinearizedBE} as
    \begin{equation}
    \nu(v)f+v\cdot\nabla_xf=K(f).
    \end{equation}
    Observing that the integral operator $K$ can serve as an agent of averaging, it is natural to imagine applying velocity averaging lemma to linearized Boltzmann equation.  In case the source term $\Psi(x,v)$ is imposed, i.e.,
    \begin{equation}
    \nu(v)f+v\cdot\nabla_xf=K(f)+\Psi(x,v),
    \end{equation}
    one can derive an integral equation
    
    \begin{equation}\begin{split}
    f(x,v)&=\int_0^\infty e^{-\nu(v)t}[K(f)(x-vt,v)+\Psi(x-vt,v)]\,dt\\
    &=:S(K(f)+\Psi)\\&=SK(f)+S(\Psi),\end{split}
    \end{equation}
    where 
    \begin{equation} \label{eq:SDef}
    S(h)(x,v):= \int^\infty_0 e^{-\nu(v)t} h(x-vt,t)\, dt.
    \end{equation}
    \noindent    Performing the  Picard iteration, formally we can derive that
    \begin{equation}
    f=\sum_{k=0}^\infty S(KS)^k(\Psi).
    \end{equation}
    By carefully adapting the idea of velocity averaging lemma, we find every two iterations improve regularity in space of order $\frac12$. 
    More precisely, the following lemma was proved in \cite{chuang19}.
    \begin{lem} \label{lemma:ksk}
        The operator $KSK:L^2_v(\R^3;\tilde{H}^{s}_x(\R^3)) \to L^2_v(\R^3;\tilde{H}^{s+\frac{1}{2}}_x(\R^3))$ is bounded for any $s\geq 0$. 
    \end{lem}
    \noindent Here, the mixed fractional Sobolev space is defined as follows.
    
    \begin{defi}\label{SobolveFourier}
    We say $u:\mathbb{R}^3\times\mathbb{R}^3\to \mathbb{R}$ is in $L^2_v(\R^3;\tilde{H}^s_x(\R^3)) $ if
    \begin{equation}
    \|u\|_{L^2_v(\R^3;\tilde{H}^s_x(\R^3))}=\left(\int_{\R^3}\int_{\R^3}(1+|\xi|^2)^{s}|\F(u)(\xi,v)|^2\,d\xi dv\right)^\frac{1}{2}<\infty,
    \end{equation} 
    where $\F(u)(\xi,v)$ is the Fourier transform of $u$ with respect to space variable $x$.
    \end{defi}

    Motivated by the successful application of velocity averaging lemma to the study of regularity issue for stationary linearized Boltzmann equation in the whole space, we consider to give a similar account for the regularity problem in a bounded domain.  However, we immediately notice that Definition~\ref{SobolveFourier} does not work for bounded space because that the Fourier transform is involved. 
    For bounded domains, we adopt the fractional Sobolev space through the Slobodeckij semi-norm. 
    \begin{defi} \label{FracSobolev}
    Let $s\in (0,1)$, $\Omega \subset \mathbb{R}^3 $ open. 
    We say $f(x,v)\in L^2_v(\mathbb{R}^3;H^s_x(\Omega))$ if  $f\in L^2(\O\times\R^3)$ and
    \begin{equation}
    \int_{\mathbb{R}^3}\int_\Omega\int_\Omega\frac{|f(x,v)-f(y,v)|^2}{|x-y|^{3+2s}}\,dxdydv<\infty,
    \end{equation}
    with
    \begin{equation}
        \| f \|_{L^2_v(\mathbb{R}^3;H^s_x(\Omega))} =\left( \| f\|^2_{L^2(\O\times\R^3)}+\int_{\mathbb{R}^3}\int_\Omega\int_\Omega\frac{|f(x,v)-f(y,v)|^2}{|x-y|^{3+2s}}\,dxdydv \right)^\frac{1}{2}.
    \end{equation}
    \end{defi}
    \noindent Notice that Definition~\ref{SobolveFourier} and Definition~\ref{FracSobolev} of fractional Sobolev spaces are equivalent on the whole space. In other 
    words, for $0<s<1$, there exist two positive constants $C_1=C_1(s)$ and $C_2=C_2(s)$ such that 
    \begin{equation} \label{eq:SobolevEquivalence}
        C_1 \| u \|_{L^2_v(\mathbb{R}^3;H^s_x(\R^3))} \leq \|u\|_{L^2_v(\R^3;\tilde{H}^s_x(\R^3))}  \leq C_2  \| u \|_{L^2_v(\mathbb{R}^3;H^s_x(\R^3))} 
    \end{equation}
    for any $u\in L^2_v(\R^3;\tilde{H}^s_x(\R^3)) $.

    Here, we shall first introduce our main result and then explain the multiple obstacles we encounter and how we overcome them. 
    We consider a bounded convex domain which satisfies the  following assumption.
    
    \begin{defi} \label{defi:PositiveCurvature}
        We say a $C^2$ bounded convex domain $\Omega$ in $\R^3$ satisfies the positive curvature condition if 
        $\partial \Omega$ is of positive Gaussian curvature.
    \end{defi}
    \begin{rmk}
        Positive curvature condition implies uniform convexity, which would also imply strict convexity. If the domain is compact, 
        then its being strictly convex is equivalent to being uniformly convex. On the contrary, a uniformly convex domain does 
        not necessarily satisfy the positive curvature condition.
    \end{rmk}
    We consider the incoming boundary value problem for linearized Boltzmann equation in $\Omega$,
    
    \begin{equation} \label{eq:SLBwithBD1}
        \left\{ 
        \begin{aligned}
            v\cdot\nabla_x f(x,v) &=L(f)(x,v),&\text{for }&x\in\Omega,\, v\in\R^3, \\
            f|_{\Gamma_-}(q,v)&=g(q,v),&\text{for }&(q,v)\in\Gamma_-,
        \end{aligned}\right.
    \end{equation}
    where 
    \begin{equation*} 
    \Gamma_-:=\{(q,v)\in\partial\Omega\times\R^3:\, n(q)\cdot v<0\},  
    \end{equation*}
    and $n(q)$ is the unit outward normal of $\partial\O$ at $q$. 
    In this context, $L$ satisfies one of hard sphere, cutoff hard, and cutoff Maxwellian potentials. 
    The detailed assumption on $L$ will be addressed in Section~\ref{sec:CollisionOperator}.
    
    Regarding the  existence result of boundary value problem \eqref{eq:SLBwithBD1},  it has been studied by Guiraud \cite{guiraud70} for convex domains and by Esposito, Guo, Kim, and  Marra \cite{esposito13}  for  general domains. 
    In the paper of Esposito, Guo, Kim, and  Marra \cite{esposito13}, they proved the solution is continuous away from the grazing set.  With stronger assumption on cross-section  $B$, namely,
    
    \begin{equation} \begin{split}\label{saparableB}
         &B(|v-v_*|,\theta)= |v-v_*|^\gamma \beta(\theta),\\&0\leq\beta(\theta)\leq C\sin\theta\cos\theta,\end{split}
    \end{equation}
    the interior H\"{o}lder estimate was established in \cite{chen18}  and later improved to interior pointwise estimate for first derivatives \cite{chen19}. Recently, the nonlinear case was established in \cite{chen20} for hard sphere potential.  Notice that, in \cite{chen18, chen19,chen20}, the  fact  $K$ improves regularity in velocity is a key property used. The idea is to move the regularity in velocity to space through transport and collision. This idea was inspired by the mixture lemma by Liu and Yu \cite{liu04}.  In contrast, in the present result, we do not need the smoothing effect of $K$ in velocity; the integral operator $K$ itself  provides "velocity averaging"  and therefore regularity. 
    Regarding regularity issues for the time dependent Boltzmann equation, we refer the interested readers to \cite{guo16,guo17}.

    In this article, we assume the following two conditions on the incoming data $g$.
    \begin{assump}\label{assump:incoming}
    There are positive constants $a,C$ such that 
        \begin{equation} \label{eq:condition2}
         |g(q_1,v)|\leq C\, e^{-a|v|^2}
        \end{equation}
        and
        \begin{equation}\label{eq:condition3}
        |g(q_1,v)-g(q_2,v)|\leq C\, |q_1-q_2|, 
        \end{equation}
    for any $(q_1,v)\in\Gamma_-$ and $(q_2,v)\in\Gamma_-$.
    \end{assump}
    The main result of this paper is as follows.
    
    \begin{thm} \label{thm}
        Let bounded convex domain $\O \subset  \R^3$ satisfy the positive curvature condition in Definition~\ref{defi:PositiveCurvature} and linearized collision operator $L$ 
        satisfy angular cutoff assumption \eqref{eq:cutoff}. Suppose the incoming data $g$ satisfies Assumption~\ref{assump:incoming}. Then, for any solution $f\in L^2(\O\times \R^3)$
        to stationary linearized Boltzmann equation \eqref{eq:SLBwithBD1}, we have  
        \begin{equation}
        f \in L^2_v(\mathbb{R}^3;{H}^{1-\epsilon}_x(\O)),
        \end{equation}
        for any $0<\epsilon<1$.
    \end{thm}
    
    We shall sketch the proof and reveal the difficulties induced by geometry and the method we tackle the problem. 
    \begin{defi} \label{def:ExitTime}
        Let $x\in \Omega$ and $v\in \mathbb{R}^3$. We define
    \begin{align*}
        &\tau_-(x,v):=\inf_{t>0} \{ t:\,x-vt \notin \Omega\},\\
        &q_-(x,v):=x-\tau_-(x,v)v, \\
        &\tau_+(x,v):=\inf_{t>0} \{ t:\,x+vt \notin \Omega\},\\
        &q_+(x,v):=x+\tau_+(x,v)v.
    \end{align*}
    \end{defi}
    \noindent With the notations of Definition~\ref{def:ExitTime},  one can rewrite \eqref{eq:SLBwithBD1} as
      the integral equation
    
    \begin{equation}\begin{split}
    f(x,v)=&e^{-\nu(v) \tau_-(x,v)}g(q_-(x,v), v) \\&+\int_0^{\tau_-(x,v)}e^{-\nu(v)s}K(f)(x-sv,v)\,ds.
    \end{split}\end{equation}
    Hereafter, we define 
    \begin{align}
        (Jg)(x,v)&:=e^{-\nu(v)\tau_-(x,v)}g(q_-(x,v),v), \label{eq:jgdefinition} \\
        (S_\Omega f)(x,v)&:=\int^{\tau_-(x,v)}_0 e^{-\nu(v)s}f(x-sv,v)ds.
    \end{align}
    Notice that $\SO:L^p(\O\times\R^3)\to L^p(\O\times\R^3)$ and $J:L^p(\Gamma_-;d\sigma)\to L^p(\O\times\R^3)$ are bounded for $1\leq p \leq \infty$
    with
    \[
        d\sigma = |v\cdot n(q)|\,d\Sigma(q)dv,
    \]
    where $\Sigma(q)$ is the surface element on $\partial\O$ at $q$. Performing Picard iteration, we have
    \begin{equation}\begin{split}
    f(x,v)=&J(g)+S_{\Omega}K(f)\\=&J(g)+S_{\Omega}KJ(g)+S_{\Omega}KS_{\Omega}K(f)\\=&J(g)+S_{\Omega}KJ(g)+S_{\Omega}KS_{\Omega}KJ(g)+S_{\Omega}KS_{\Omega}KS_{\Omega}K(f)\\=&J(g)+S_{\Omega}KJ(g)+S_{\Omega}KS_{\Omega}KJ(g)+S_{\Omega}KS_{\Omega}KS_{\Omega}KJ(g)\\&+S_{\Omega}KS_{\Omega}KS_{\Omega}KS_{\Omega}K(f)\\= &g_0+g_1+g_2+g_3+f_4,
    \end{split}\end{equation}
    where
    \begin{align}
    g_i&:=(S_\Omega K)^iJ(g), \label{eq:giDef}\\
    f_i&:=(S_\Omega K)^i(f).  \label{eq:fiDef}
    \end{align}
    
    We observe that each $g_i$ is directly under influence of boundary data and the geometry of the domain. Our strategy is to prove $ g_i\in L^2_v(\mathbb{R}^3;{H}^{1-\epsilon}_x(\mathbb{R}^3))$. And, concerning the remaining term $f_4$,  we shall match up the regularity of boundary terms. 
    
    We point out the difference between the cases for the whole space and a bounded domain. 
    Let $h$ be any measurable function defined in $\O\times\R^3$. We use the notation $\widetilde{h}$ to denote its zero extension in 
$\R^3\times\R^3$. And, let $Z:\O\times\R^3\to \R^3\times\R^3$ be the zero extension operator from $\O\times\R^3$ to $\R^3\times\R^3$, namely, 
\begin{equation} \label{eq:ZeroExtension}
    (Zh)(x,v)=\widetilde{h}(x,v)=
    \begin{cases}
        h(x,v),&\quad\text{if }x\in\O,\\
        0,&\quad\text{otherwise.}
    \end{cases}
\end{equation}
    Suppose $f\in L^2(\O\times\R^3)$. Notice that 
    \begin{equation} \label{1iteration}
    \left. SK(\widetilde{f})\right|_{\Omega}=S_{\Omega}K(f). 
    \end{equation}
    Therefore, applying Lemma~\ref{lemma:ksk}, we have
    \begin{cor} \label{cor:ksok}
        The operator $K\SO K:L^2(\O\times\R^3)\to L^2_v(\R^3;H^{\sfrac 1 2 }_x(\O))$ is bounded.
    \end{cor}
    
    \begin{rmk}
        Corollary~\ref{cor:ksok} can be viewed as a variant of Theorem~4 in \cite{golse88}. For the latter, the transport equation under consideration is, for given $h\in L^2(\O\times\R^3)$, 
        \begin{equation} 
                u(x,v) + v\cdot\nabla_x u(x,v) =h(x,v),\, \text{for }x\in\Omega,\, v\in\R^3.
        \end{equation}
        On the other hand, the transport equation we consider here is, for given $h\in L^2(\O\times\R^3)$,
        \begin{equation} 
            \nu(v)u(x,v) + v\cdot\nabla_x u(x,v) =Kh(x,v),\, \text{for }x\in\Omega,\, v\in\R^3.
        \end{equation}
    \end{rmk}
    \noindent However, if we want to further iterate, we have to take the geometric structure into consideration. 
    As mentioned earlier, the method of Fourier transform does not apply to bounded domains. 
    In this article, we adopt Slobodeckij semi-norm as an alternative concept of Sobolev function class. 
    As a result, we have 
    
    \begin{lem} \label{lemma:soksok}
        The operator $\SO K\SO K:L^2(\O\times\R^3)\to L^2_v(\R^3;H^{\sfrac 1 2}_x(\O))$ is bounded.
    \end{lem}
    
    \noindent Therefore, if we end iteration at $f_2$, we can already claim $f\in L^2_v(\mathbb{R}^3; H^{\sfrac12}_x(\Omega))$.  
    
    Considering piling up the regularity, we notice that 
    \begin{equation}
    \left. SKSK(\widetilde{f})\right|_{\Omega}\neq S_{\Omega}KS_{\Omega}K(f).
    \end{equation} 
    It seemingly comes to the limit of this strategy. 
    Surprisingly, we have  
    \begin{lem} \label{lemma:soksokextension}
        $Z\SO K\SO K:L^2(\O \times \R^3)\to L^2_v(\R^3;H^{\frac{1}{2}-\epsilon}_x(\R^3))$ is bounded for any 
        $\epsilon\in (0,\frac 1 2)$. Furthermore, there is a constant $C$ independent of $\epsilon$ and $f$ such that 
        \begin{equation}
            \| \widetilde{\SO K \SO Kf}\|_{L^2_v(\R^3;H^{\frac{1}{2}-\epsilon}_x(\R^3))} \leq \frac{C}{\sqrt{\epsilon}}\, \| f \|_{ L^2(\O\times \R^3)}.
        \end{equation}
    \end{lem}

    \noindent That is, zero extension only reduces infinitesimal regularity. Therefore, after zero extension, we can repeat our strategy and obtain the desired result.
    
    The rest of the article is organized as follows. 
    Section~\ref{sec:BoltzmannEquation} gives a brief overview of the Boltzmann equation and the details of our setting. 
    In Section~\ref{sec:CollisionOperator}, we recall several basic properties of the linearized collision operator. 
    In Section~\ref{sec:VelAvgLBE}, we recapitulate the velocity averaging for linearized Boltzmann equation in the whole space.
    Section~\ref{sec:geometry}  provides some geometric properties for bounded convex domains that will be used in our estimates. 
    In Section~\ref{sec:transport}, we study the regularity of transport equation in bounded convex domains.  
    Section~\ref{sec:averaging} is devoted to the regularity via velocity averaging.   

%%%%%%%%%%%%%%%%%%%%%%%%%%%%%%%%%%%%%%%%%%%%%%%%%%%%%%%%%%%%%%%%%%%%%%%%%%%%%%%%%%%%%%%%%%%%%%%%%%%%%%%%%%%%%%%%%%%%%%%%%%%%%%%%%%%%%%%%%%%%%%%%%%%%%%%%%%%%%%%%%%%%%%%%%%%%%%%%%%%%%%%%%%%%%%%%%%%%%%%%%%%%%%%%%%%%%%%%%%%%%%%%%%%%%%%%%%%%%%%%%%%%%%%%%%%%%%%%%%%%%%%%%%%%%%%%%%%%%%%%%%%%%%%%%%%%%%%%%%%%%%%%%%%%%%%%%%%%%%%%%%%%%%%%%%%%%%%%%%%%%%%%%%%%%%%%%%%%
\section{Boltzmann Equation} \label{sec:BoltzmannEquation}%%%%%%%%%%%%%%%%%%%%%%%%%%%%%%%%%%%%%%%%%%%%%%%%%%%%%%%%%%%%%%%%%%%%%%%%%%%%%%%%%%%%%%%%%%%%%%%%%%%%%%%%%%%%%%%%%%%%%%%%%%%%%%%%%%%%%%%%%%%%%%%%%%%%%%%%%%%%%%%%%%%%%%%%%%%%%%%%%%%%%%%%%%%%%%%%%%%%%%%%%%%%%%%%%%%%%%%%%%%%%%%%%%%%%%%%%%%%%%%%%%%%%%%%%%%%%%%%%%%%%%%%%%%%%%%%%%%%%%%%%%%%%%%%%%%%%%%%%%%%%%%%%%%%%%%%%%%%%%%%
%%%%%%%%%%%%%%%%%%%%%%%%%%%%%%%%%%%%%%%%%%%%%%%%%%%%%%%%%%%%%%%%%%%%%%%%%%%%%%%%%%%%%%%%%%%%%%%%%%%%%%%%%%%%%%%%%%%%%%%%%%%%%%%%%%%%%%%%%%%%%%%%%%%%%%%%%%%%%%%%%%%%%%%%%%%%%%%%%%%%%%%%%%%%%%%%%%%%%%%%%%%%%%%%%%%%%%%%%%%%%%%%%%%%%%%%%%%%%%%%%%%%%%%%%%%%%%%%%%%%%%%%%%%%%%%%%%%%%%%%%%%%%%%%%%%%%%%%%%%%%%%%%%%%%%%%%%%%%%%%%%%%%%%%%%%%%%%%%%%%%%%%%%%%%%%%%%%%
    The Boltzmann equation reads
    \begin{equation} \label{eq:B}
        \partial_t F +v\cdot \nabla_x F = Q(F,F).
    \end{equation}
    Here, $F=F(t,x,v)\geq 0$ is the distribution function for the particles located in the position $x$ with velocity $v$ at time $t$. 
    The collision operator $Q$ is defined as 
    \begin{equation}
        Q(F,G)=\int\limits_{\R^3}\int\limits^{2\pi}_0\int\limits^{\sfrac{\pi}{2}}_0 \biggl(F(v')G(v'_*)-F(v)G(v_*) \biggr)B(|v-v_*|,\theta)\, d\theta d\epsilon dv_*,
    \end{equation}
    where $v'$ and $v'_*$ are the velocities after the elastic collision of two particles whose velocities are 
    $v$ and $v_*$, respectively, before the encounter. Here, the cross-section $B$ is chosen according to the type of interaction between particles. 
    The precise form of cross-section $B$ depends on 
    the model that is being studied. To specify the properties of the cross-section, we adopt the following coordinates. 
    We set
    \[
        e_1=\frac{v_*-v}{|v_*-v|}
        \]
    and choose $e_2\in\S^2$ and $e_3\in\S^2$ such that $\{e_1,e_2,e_3\}$ forms an 
    orthonormal basis for $\R^3$, and define
    \[
        \alpha=\cos \theta e_1 +\sin\theta \cos\epsilon e_2+\sin\theta\sin\epsilon e_3.
        \]
    Then,
    \begin{align}
        v'&=v+\bigl( (v_*-v)\cdot\alpha \bigr)\alpha, \\
        v'_*&=v_*-\bigl( (v_*-v)\cdot\alpha \bigr)\alpha.
    \end{align}
    Throughout this article, regarding the cross-section, we apply Grad's angular cutoff potential \cite{grad63} by assuming
    \begin{equation} \label{eq:cutoff}
        0\leq B(|v-v_*|,\theta)\leq C|v-v_*|^\gamma \cos\theta \sin \theta,
    \end{equation}
    where, as mentioned above, $\gamma$ depends on the model being studied. Our discussion includes hard sphere model 
    ($\gamma = 1$), cutoff hard potential ($0<\gamma<1$), and cutoff Maxwellian molecular gases ($\gamma=0$).  
    Consider the stationary solution $F=F(x,v)$ as a perturbation of the standard Maxwellian,
    \[
        M(v)=\pi^{-\frac{3}{2}}e^{-|v|^2},
        \]
    in the form
    \begin{equation} \label{eq:linearizedF}
        F=M+M^{\frac 1 2}f.
    \end{equation}
    Plugging the expression \eqref{eq:linearizedF} into \eqref{eq:B} and discarding the nonlinear term, we arrive at the 
    stationary linearized Boltzmann equation
    \begin{equation}
        v\cdot \nabla_x f(x,v) = L(f)(x,v),
    \end{equation}
    with linearized collision operator $L$, which reads 
    \begin{equation}
        L(f)=M^{-\frac 1 2}\bigl(Q(M,M^{\frac 1 2}f)+Q(M^{\frac 1 2}f,M) \bigr).
    \end{equation}
    Under the assumption \eqref{eq:cutoff}, $L$ can be decomposed into a multiplicative operator and an integral operator, see \cite{grad63},
    \begin{equation} \label{eq:linearizedL}
        L(f)=-\nu(v)f+K(f).
    \end{equation}
    Here, $\nu$ is a function of velocity variable $v$ behaving like $(1+|v|)^\gamma$, i.e., there exist two positive constants 
    $\nu_0$ and $\nu_1$, depending only on $\gamma$, such that
    \begin{equation} \label{eq:nu}
        0<\nu_0(1+|v|)^\gamma<\nu(v)<\nu_1(1+|v|)^\gamma,
    \end{equation}
    for all $v\in\R^3$. The integral operator $K$ reads
    \[
        K(f)(x,v)=\int_{\R^3}f(x,v_*)k(v_*,v)\, dv_*,
    \]
    where the collision kernel $k$ is symmetric, that is, $k(v,v_*)=k(v_*,v)$. 
    Notice that the assumption of the cross-section here is different from and more general 
    than that in \cite{chen15,chen19}. The significant difference is that the operator $K$ in the case we consider 
    does not guarantee to have regularity in velocity variables.

    Under the decomposition \eqref{eq:linearizedL}, we then consider the boundary value problem 
    \begin{equation} \label{eq:SLBwithBD}
        \left\{ 
        \begin{aligned}
            \nu(v)f(x,v)+v\cdot\nabla_x f(x,v) &=K(f)(x,v),&\text{for }&x\in\Omega,\, v\in\R^3, \\
            f|_{\Gamma_-}(q,v)&=g(q,v),&\text{for }&(q,v)\in\Gamma_-.
        \end{aligned}\right.
    \end{equation}
    \noindent Let us pause here to make clear that how trace is being defined. We follow \cite{cessenat84,cessenat85,kawagoe18}. For $f\in L^p(\O\times\R^3)$ satisfying the equation 
        \begin{equation}
            \nu(v)f(x,v)+v\cdot\nabla_x f(x,v) =K(f)(x,v)
        \end{equation}
        in $\O\times\R^3$, since $K(f)\in L^p(\O\times\R^3)$ according to Proposition~\ref{prop:KLp},  we have
        \begin{equation}
            v\cdot\nabla_x f(x,v) =-\nu(v)f(x,v) + K(f)(x,v) \in L^p_x(\O)
        \end{equation}
        for a.e. $v\in\R^3$. For such $v\in\R^3$, we fix $q\in \Gamma_-(v)$, where 
        \[
            \Gamma_-(v):=\{ q\in \partial\O:\, (q,v)\in \Gamma_-  \},
        \]
        and consider the function 
        \begin{equation}
            h_{q,v}(s)=f(q+sv,v),
        \end{equation}
        for $0<s<\tau_+(q,v)$, where $\tau_+(q,v)$ is as defined in Definition~\ref{def:ExitTime}. Noticing that 
        \begin{equation}
            \frac{dh_{q,v}(s)}{ds}   = v\cdot\nabla_x f(q+sv,v),
        \end{equation}
        and the formulas for change of variables,
        \begin{align}
            \int_{\O} |f(x,v)|^p \, dx &= \int_{\Gamma_-(v)}\int^{\tau_+(q,v)}_0 |f(q+sv,v)|^p  |v\cdot n(q)| \, dsd\Sigma(q) \\
            \int_{\O} |v\cdot\nabla_x f(x,v)|^p \, dx &= \int_{\Gamma_-(v)}\int^{\tau_+(q,v)}_0 |\frac{\partial}{\partial s}f(q+sv,v)|^p |v\cdot n(q)| \, dsd\Sigma(q),
        \end{align}
        we obtain $h_{q,v}\in W^{1,p}_s (0,\tau_+(q,v))$ for a.e. $(q,v)\in\Gamma_-$, which would imply $h_{q,v}(s)$ is an 
        absolutely continuous function on $(0,\tau_+(q,v))$ after possibly being redefined on a set of measure zero. As a result, we define  
        \begin{equation}
            f(q,v):=\lim_{s\to 0^+}h_{q,v}(s)= \lim_{s\to 0^+} f(q+sv,v),
        \end{equation}
        for $(q,v)\in \Gamma_-.$ 
%%%%%%%%%%%%%%%%%%%%%%%%%%%%%%%%%%%%%%%%%%%%%%%%%%%%%%%%%%%%%%%%%%%%%%%%%%%%%%%%%%%%%%%%%%%%%%%%%%%%%%%%%%%%%%%%%%%%%%%%%%%%%%%%%%%%%%%%%%%%%%%%%%%%%%%%%%%%%%%%%%%%%%%%%%%%%%%%%%%%%%%%%%%%%%%%%%%%%%%%%%%%%%%%%%%%%%%%%%%%%%%%%%%%%%%%%%%%%%%%%%%%%%%%%%%%%%%%%%%%%%%%%%%%%%%%%%%%%%%%%%%%%%%%%%%%%%%%%%%%%%%%%%%%%%%%%%%%%%%%%%%%%%%%%%%%%%%%%%%%%%%%%%%%%%%%%%%%%%%
\section{Properties of the linearized collision operator} \label{sec:CollisionOperator}%%%%%%%%%%%%%%%%%%%%%%%%%%%%%%%%%%%%%%%%%%%%%%%%%%%%%%%%%%%%%%%%%%%%%%%%%%%%%%%%%%%%%%%%%%%%%%%%%%%%%%%%%%%%%%%%%%%%%%%%%%%%%%%%%%%%%%%%%%%%%%%%%%%%%%%%%%%%%%%%%%%%%%%%%%%%%%%%%%%%%%%%%%%%%%%%
%%%%%%%%%%%%%%%%%%%%%%%%%%%%%%%%%%%%%%%%%%%%%%%%%%%%%%%%%%%%%%%%%%%%%%%%%%%%%%%%%%%%%%%%%%%%%%%%%%%%%%%%%%%%%%%%%%%%%%%%%%%%%%%%%%%%%%%%%%%%%%%%%%%%%%%%%%%%%%%%%%%%%%%%%%%%%%%%%%%%%%%%%%%%%%%%%%%%%%%%%%%%%%%%%%%%%%%%%%%%%%%%%%%%%%%%%%%%%%%%%%%%%%%%%%%%%%%%%%%%%%%%%%%%%%%%%%%%%%%%%%%%%%%%%%%%%%%
In this section, we review some basic properties of the linearized collision operator $L$, see \cite{caflisch80,grad63}. 
As mentioned earlier, under the assumption \eqref{eq:cutoff}, 
$L$ can be decomposed into a multiplicative operator and an integral operator,
\begin{equation} \label{eq:linearizedL2}
    L(f)=-\nu(v)f+K(f).
\end{equation}
There are two constants $0<\nu_0<\nu_1$ such that the following inequality holds for all $v\in \R^3$,
\begin{equation} \label{eq:nu2}
    0<\nu_0\leq \nu_0(1+|v|)^\gamma<\nu(v)<\nu_1(1+|v|)^\gamma.
\end{equation}
The integral operator $K$ is defined as
\begin{equation*}
    K(f)(x,v)=\int_{\R^3}f(x,v_*)k(v_*,v)\, dv_*,
\end{equation*}
and the collision kernel $k$ is symmetric. Furthermore, by Caflisch \cite{caflisch80}, we have an upper bound for $k$,
\begin{equation} \label{eq:k}
    |k(v,v_*)|\leq C \frac 1 {|v-v_*|}(1+|v|+|v_*|)^{-(1-\gamma)}e^{-\frac 1 8 \left(|v-v_*|^2 + \frac{(|v|^2-|v_*|^2)^2}{|v-v_*|^2} \right)},
\end{equation}
where $C>0$ is a constant depending only on $\gamma$. The following lemma by Caflisch \cite{caflisch80} is crucial in our estimates.
\begin{lem} \label{lemma:Cafineq}
    For any positive constants $\epsilon$, $a_1$ and $a_2$, there exists $C=C(\epsilon,a_1,a_2)$ depending only on $\epsilon$, $a_1$ and $a_2$ such that
        \begin{align*}
        \int_{\R^3} \frac{1}{|v-v_*|^{3-\epsilon}}e^{-a_1|v-v_*|^2 - a_2 \frac{(|v|^2-|v_*|^2)^2}{|v-v_*|^2}}\,dv_* &\leq C\,\frac{1}{1+|v|} \\
        &\leq C.
        \end{align*}
\end{lem}
\noindent The above inequality combining with \eqref{eq:k} immediately implies the following facts.
\begin{cor} \label{cor:k}
    For any $v\in\R^3$,
    \begin{align}
        \int_{\R^3} |k(v,v_*)| \, dv_* &\leq C\, \left(  \frac 1 {1+|v|} \right)^{2-\gamma} \leq C, \label{eq:kl1} \\
        \int_{\R^3} |k(v,v_*)|^2 \, dv_* &\leq C \, \left(  \frac 1 {1+|v|} \right)^{3-2\gamma}  \leq C. \label{eq:kl2}
    \end{align}
\end{cor}

\begin{prop} \label{prop:KLp}
    The integral operator $K:L^p_v(\R^3)\to L^p_v(\R^3)$ is bounded for $1\leq p \leq \infty$.
\end{prop}
\begin{proof}
    If $1<p<\infty$, for given $h\in L^p_v(\R^3)$, we have 
    \begin{equation} \begin{split}
        | Kh(v) |^p &=\left( \int_{\R^3} h(v_*)k(v_*,v)\, dv_* \right)^p \\
        &\leq \biggl( \int_{\R^3} |h(v_*)||k(v_*,v)|^{\frac{1}{p}}|k(v_*,v)|^{\frac{1}{p'}}\,dv_* \biggr)^p  \\
        &\leq  \left(  \int_{\R^3}  |h(v_*)|^p|k(v_*,v)|\, dv_* \right) \left( \int_{\R^3} |k(v_*,v)|\, dv_* \right)^{\frac{p}{p'}}  \\
        &\leq C  \int_{\R^3}  |h(v_*)|^p |k(v_*,v)|\, dv_*,
    \end{split}\end{equation}
    where we have applied Corollary~\ref{cor:k} in the above derivation. Therefore, applying Corollary~\ref{cor:k} again, we obtain 
    \begin{equation} \begin{split}
        \int_{\R^3} |Kh(v)|^p\, dv&\leq C\int_{\R^3} \int_{\R^3}  |h(v_*)|^p|k(v_*,v)|\,dv_*dv \\
        &=C \int_{\R^3} \int_{\R^3} \big( |k(v_*,v)|\,dv \big) |h(v_*)|^p\,dv_* \\
        &\leq C  \int_{\R^3}  |h(v_*)|^p dv_*.
    \end{split}\end{equation}
    The proof for the cases where $p=1,\infty$ is straightforward. We should omit it.
\end{proof}
\noindent We see from Proposition~\ref{prop:KLp} $K$ is a bounded operator from $L^p_v(\R^3)$ to $L^p_v(\R^3)$. In particular, we 
shall only use the case $p=2$.
\begin{prop} \label{prop:1overvk}
    For any $v\in\R^3$ and $\epsilon\in(0,2)$, there is a constant $C$ independent of $v$ such that 
    \begin{equation}
        \int_{\R^3}\frac{1}{|v_*|^{2-\epsilon}}\,|k(v,v_*)|\,dv_* \leq C.
    \end{equation}
\end{prop}
\begin{proof}
    Let $v\in\R^3$ and $\epsilon\in(0,2)$ be given. From \eqref{eq:k}, we have 
    \[
        |k(v,v_*)|\leq C\, \frac{1}{|v-v_*|}e^{-\frac 1 8 |v-v_*|^2}.
    \]
    It therefore suffices to prove that 
    \begin{equation} \label{eq:1overvk}
        \int_{\R^3}\frac{1}{|v_*|^{2-\epsilon}|v-v_*|}e^{-\frac 1 8 |v-v_*|^2}\,dv_*\leq C
    \end{equation}
    for some $C$. If $|v_*|\leq |v-v_*|$, then 
    \[
        \frac{1}{|v_*|^{2-\epsilon}|v-v_*|} \leq \frac{1}{|v_*|^{3-\epsilon}}.
    \]
    If $|v_*| > |v-v_*|$, then 
    \[
        \frac{1}{|v_*|^{2-\epsilon}|v-v_*|} \leq \frac{1}{|v-v_*|^{3-\epsilon}}.
    \]
    In any case,
    \begin{equation}\begin{split}
        &\int_{\R^3}\frac{1}{|v_*|^{2-\epsilon}|v-v_*|}e^{-\frac 1 8 |v-v_*|^2}\,dv_* \\
        &\quad \leq \int_{\R^3}\frac{1}{|v_*|^{3-\epsilon}}e^{-\frac 1 8 |v-v_*|^2}\,dv_* + \int_{\R^3}\frac{1}{|v-v_*|^{3-\epsilon}}e^{-\frac 1 8 |v-v_*|^2}\,dv_*. 
    \end{split}\end{equation}
    Clearly,
    \[
        \int_{\R^3}\frac{1}{|v-v_*|^{3-\epsilon}}e^{-\frac 1 8 |v-v_*|^2}\,dv_* = \int_{\R^3}\frac{1}{|w|^{3-\epsilon}}e^{-\frac 1 8 |w|^2}\,dw \leq C. 
    \]
    Notice that 
    \begin{equation} \begin{split}
        &\int_{\R^3}\frac{1}{|v_*|^{3-\epsilon}}e^{-\frac 1 8 |v-v_*|^2}\,dv_* \\
        &\quad \leq\int_{\{|v_*|\leq 1\}}\frac{1}{|v_*|^{3-\epsilon}}e^{-\frac 1 8 |v-v_*|^2}\,dv_*+ \int_{\{ |v_*|>1\}}\frac{1}{|v_*|^{3-\epsilon}}e^{-\frac 1 8 |v-v_*|^2}\,dv_* \\
        &\quad \leq\int_{\{|v_*|\leq 1\}}\frac{1}{|v_*|^{3-\epsilon}}\,dv_* +\int_{\R^3}e^{-\frac 1 8 |v-v_*|^2}\,dv_* \\
        &\quad \leq C.
    \end{split}\end{equation}
    We deduce \eqref{eq:1overvk}.
\end{proof}
In Section~\ref{sec:transport}, we will need the following estimate. The case for hard sphere is proved in \cite{chen14}. More generally, the proof therein can be 
extended to the cases for cutoff hard and cutoff Maxwellian potentials.
\begin{lem} \label{lemma:MaxwellianBound}
    Suppose that a function $h$ satisfies the following estimate for some constant 
    $a \in [0,\frac 1 4)$,
    \begin{equation}
        |h(v)|\leq e^{-a|v|^2}.
    \end{equation}
    Then there exists a positive constant $C_a$ such that
    \begin{equation}
        |K(h)(v)|\leq C_a e^{-a|v|^2},
    \end{equation}
    where $C_a$ is a constant depending on $a$.
\end{lem}
%%%%%%%%%%%%%%%%%%%%%%%%%%%%%%%%%%%%%%%%%%%%%%%%%%%%%%%%%%%%%%%%%%%%%%%%%%%%%%%%%%%%%%%%%%%%%%%%%%%%%%%%%%%%%%%%%%%%%%%%%%%%%%%%%%%%%%%%%%%%%%%%%%%%%%%%%%%%%%%%%%%%%%%%%%%%%%%%%%%%%%%%%%%%%%%%%%%%%%%%%%%%%%%%%%%%%%%%%%%%%%%%%%%%%%%%%%%%%%%%%%%%%%%%%%%%%%%%%%%%%%%%%%%%%%%%%%%%%%%%%%%%%%%%%%%%%%%
\section{Velocity averaging for linearized Boltzmann equation in the whole space} \label{sec:VelAvgLBE}%%%%%%%%%%%%%%%%%%%%%%%%%%%%%%%%%%%%%%%%%%%%%%%%%%%%%%%%%%%%%%%%%%%%%%%%%%%%%%%%%%%%%%%%%%%%%%%%%%%%%%%%%%%%%%%%%%%%%%%%%%%%%%%%%%%%%%%%%%%%%%%%%%%%%%%%%%%%%%%%%%%%%%%%%%%%%%%%%%%%%%%%%%%%%%%%
%%%%%%%%%%%%%%%%%%%%%%%%%%%%%%%%%%%%%%%%%%%%%%%%%%%%%%%%%%%%%%%%%%%%%%%%%%%%%%%%%%%%%%%%%%%%%%%%%%%%%%%%%%%%%%%%%%%%%%%%%%%%%%%%%%%%%%%%%%%%%%%%%%%%%%%%%%%%%%%%%%%%%%%%%%%%%%%%%%%%%%%%%%%%%%%%%%%%%%%%%%%%%%%%%%%%%%%%%%%%%%%%%%%%%%%%%%%%%%%%%%%%%%%%%%%%%%%%%%%%%%%%%%%%%%%%%%%%%%%%%%%%%%%%%%%%%%%
Lemma~\ref{lemma:ksk}, which was first investigated in \cite{chuang19}, plays an important role in our analysis. For readers' convenience, 
we recapitulate the proof of Lemma~\ref{lemma:ksk} in this section. The idea of Lemma~\ref{lemma:ksk} originated from the 
velocity averaging lemma in \cite{golse88}, which can be proved by the method of Fourier transform. Roughly speaking, the $L^2$ averaging lemma 
can be stated as that if $u\in L^2(\R^n\times\R^n)$ is a solution to the transport equation in $\R^n\times\R^n$ with a source term $f\in L^2(\R^n\times\R^n)$,
\begin{equation}
    u+v\cdot\nabla_x u =f,
\end{equation}
then, for any $\psi\in L^\infty_c (\R^n)$, a bounded and compactly supported function, the velocity average of $u$ satisfies
\begin{equation} \label{eq:L2velavg}
    \int_{\R^n} u(\cdot,v)\psi(v)\,dv\in \tilde{H}^{\sfrac 1 2}(\R^n).
\end{equation}
By analogy with the above result, we consider the transport equation in $\R^3\times\R^3$ with a source term $K(f)\in L^2(\R^3\times\R^3)$,
\begin{equation} \label{eq:Btransport}
    \nu(v)u+v\cdot\nabla_x u =K(f)
\end{equation}
(Notice that $f\in L^2(\R^3\times\R^3)$ implies $K(f)\in L^2(\R^3\times\R^3)$ according to Proposition~\ref{prop:KLp}). Recall from \eqref{eq:SDef}
\begin{equation}
    (Sf)(x,v):=\int^\infty_0 e^{-\nu(v)s}f(x-sv,v)\,ds.
\end{equation}
A direct calculation shows that $S$ is a bounded operator from $L^p(\R^3\times\R^3)$ to $L^p(\R^3\times\R^3)$ for $1\leq p \leq\infty$ and 
the $L^2$ solution of \eqref{eq:Btransport} is expressed as 
\begin{equation}
    u(x,v)=(SK)(f)(x,v).
\end{equation}
If we regard the integral operator $K$ as a kind of velocity averaging, then it is natural to anticipate $K(u)=KSK(f)$ has certain regularity 
in the space variable in view of \eqref{eq:L2velavg}. And it turns out we indeed have Lemma~\ref{lemma:ksk}. In the following proof, 
we adopt the idea of Fourier transform in \cite{golse88} with a careful application of 
Cauchy-Schwarz inequality. Furthermore, we also take advantage of \eqref{eq:nu2} for the function $\nu$ to gain some integrability 
(see \eqref{eq:nusplit}) since the support of $k(v,v_*)$ is not bounded. We also remark that Lemma~\ref{lemma:ksk} holds for any function $\nu$ satisfying 
\eqref{eq:nu2} and any kernel $k$ satisfying \eqref{eq:kl1} and \eqref{eq:kl2}.
\begin{proof}[Proof of Lemma~\ref{lemma:ksk}]
    We denote the Fourier transform in $x$ of a function $h$ by $\F h$, the 
    corresponding variable by $\xi$. Since the operator $K$ does not involve the space vairable, a direct calculation shows that, 
    for $f\in L^1(\R^3\times\R^3)$, we have
    \begin{align}
        \F(Sf)(\xi,v)=&\frac{\F f(\xi,v)}{\nu(v)+i(v\cdot\xi)}, \label{eq:FourierRelation1}\\
        \F(Kf)(\xi,v)=&K(\F f)(\xi,v).  \label{eq:FourierRelation2}
    \end{align}
    For the case where $f\in L^2(\R^3\times\R^3)$, one can approximate $f$ by a sequence of functions in $L^1\cap L^2(\R^3\times\R^3)$ to 
    obtain \eqref{eq:FourierRelation1}-\eqref{eq:FourierRelation2}. To prove Lemma~\ref{lemma:ksk}, it suffices to show the 
    boundedness of integral
    \begin{equation} \label{eq:KSKintegral}
        I:=\int_{\R^3}\int_{\R^3}|\xi|^{2s+1}\left| \F KSKf(\xi,v)  \right|^2\,d\xi dv
    \end{equation}
    for given $f\in L^2_v(\R^3;\tilde{H}^{s}_x(\R^3))$. Using identities \eqref{eq:FourierRelation1}, \eqref{eq:FourierRelation2} and 
    Cauchy-Schwarz inequality yields 
    \begin{equation}\begin{split}
        I&=\int_{\R^3}\int_{\R^3}|\xi|^{2s+1}\left| \int_{\R^3} \F SKf(\xi,v_*)k(v_*,v)\,dv_* \right|^2\,d\xi dv \\
        &=\int_{\R^3}\int_{\R^3}|\xi|^{2s+1}\left| \int_{\R^3} \frac{\F Kf(\xi,v_*)}{\nu(v_*)+i(v_*\cdot\xi)}k(v_*,v)\,dv_* \right|^2\,d\xi dv \\
        &\leq \iint  |\xi|^{2s+1} \left( \int |\F Kf(\xi,v_*)|^2(1+|v_*|)^{3-2\gamma}|k(v_*,v)|\, dv_*  \right) \\
        &\quad \times \left( \int \frac{|k(v_*,v)|}{\bigl(\nu(v_*)^2+(v_*\cdot\xi)^2\bigr)(1+|v_*|)^{3-2\gamma}}\,dv_*\right)d\xi dv. 
    \end{split}\end{equation}
    By Corollary~\ref{cor:k}, we have 
    \begin{equation}\begin{split}
        |\F Kf(\xi,v_*)|^2 =& \left| \int_{\R^3} \F f(\xi,w)k(w,v_*)\,dw   \right|^2 \\
        \leq& \left(\int_{\R^3}|\F f(\xi,w)|^2\,dw\right)\left(\int_{\R^3}|k(w,v_*)|^2 \,dw\right) \\
        \leq& C\, \frac{\int_{\R^3}|\F f(\xi,w)|^2\,dw}{(1+|v_*|)^{3-2\gamma}},
    \end{split}\end{equation}
    which leads to
    \begin{equation}\begin{split}
        \int |\F Kf(\xi,v_*)|^2(1+|v_*|)^{3-2\gamma}|k(v_*,v)|\, dv_* \leq& C\, \iint |\F f(\xi,w)|^2 |k(v_*,v)|\, dwdv_*  \\
        \leq& C\, \int |\F f(\xi,w)|^2 \,dw, 
    \end{split}\end{equation}
    where the second inequality follows from Corollary~\ref{cor:k}. Therefore, it follows that 
    \begin{equation}\begin{split}
        I\leq&C\, \iiiint |\xi|^{2s+1} |\F f(\xi,w)|^2 \frac{|k(v_*,v)|}{\bigl(\nu(v_*)^2+(v_*\cdot\xi)^2\bigr)(1+|v_*|)^{3-2\gamma}}\, dvdv_* d\xi dw \\
        \leq& \iiint |\xi|^{2s+1} |\F f(\xi,w)|^2 \frac{1}{\bigl(\nu(v_*)^2+(v_*\cdot\xi)^2\bigr)(1+|v_*|)^{5-3\gamma}}\, dv_* d\xi dw,
    \end{split}\end{equation}
    where we have used Corollary~\ref{cor:k} again in the last line. We observe that 
    \begin{equation} \label{eq:nusplit}
        \begin{split}
        \frac{1}{\nu(v_*)^2+(v_*\cdot \xi)^2}=& \frac{1}{\bigl(\nu(v_*)^2+(v_*\cdot \xi)^2\bigr)^{2/3}}\cdot\frac{1}{\bigl(\nu(v_*)^2+(v_*\cdot \xi)^2\bigr)^{1/3}} \\
        \leq& \frac{1}{\bigl(\nu_0^2+(v_*\cdot \xi)^2\bigr)^{2/3}} \cdot \frac{1}{\nu(v_*)^{\sfrac 2 3}} \\
        \leq &C\, \frac{1}{\bigl(\nu_0^2+(v_*\cdot \xi)^2\bigr)^{2/3}} \cdot \frac{1}{(1+|v_*|)^{\sfrac{2\gamma}{3}}}.
    \end{split}\end{equation}
    As a result, we have 
    \begin{equation}
        I\leq C \iiint |\xi|^{2s+1} \frac{|\F f(\xi,w)|^2}{\bigl(\nu_0^2+(v_*\cdot \xi)^2\bigr)^{2/3}(1+|v_*|)^{5-\sfrac{7\gamma}{3}}}\,dv_*d\xi dw.
    \end{equation}
    Denote the component of $v_*$ parallel to $\xi$ and the component perpendicular to $\xi$ respectively by 
    \begin{equation} 
        \left\{ 
            \begin{array}{l}
                v_{*\parallel\xi} = v_* \cdot \frac{\xi}{|\xi|} \\
                v_{*\perp\xi} = v_* -(v_{*\parallel\xi}) \frac{\xi}{|\xi|}.
            \end{array} 
            \right.
        \end{equation}
    \noindent Consequently, we deduce
    \begin{equation}\begin{split}
        I\leq& C\, \iiiint |\xi|^{2s+1} \frac{|\F f(\xi,w)|^2}{\bigl(\nu_0^2+|\xi|^2v^2_{*\parallel\xi}\bigr)^{2/3}(1+|v_{*\perp\xi}|)^{5-\sfrac{7\gamma}{3}}}\,dv_{*\perp\xi}d_{*\parallel\xi}d\xi dw \\
        \leq& C\, \iiint |\xi|^{2s+1} \frac{|\F f(\xi,w)|^2}{\bigl(\nu_0^2+|\xi|^2v^2_{*\parallel\xi}\bigr)^{2/3}} \,dv_{*\parallel\xi}d\xi dw \\
        \leq&C\, \iint |\xi|^{2s} |\F f(\xi,w)|^2 \, d\xi dw \\
        \leq&C \| f \|^2_{L^2_v(\R^3;\tilde{H}^{s}_x(\R^3))},
    \end{split}\end{equation}
    where the second inequality follows since $5-\frac{7\gamma}{3} >2$ for $0\leq \gamma \leq 1$ and $v_{*\perp\xi} $ is two-dimensional plane. 
    This completes the proof.
\end{proof}
%%%%%%%%%%%%%%%%%%%%%%%%%%%%%%%%%%%%%%%%%%%%%%%%%%%%%%%%%%%%%%%%%%%%%%%%%%%%%%%%%%%%%%%%%%%%%%%%%%%%%%%%%%%%%%%%%%%%%%%%%%%%%%%%%%%%%%%%%%%%%%%%%%%%%%%%%%%%%%%%%%%%%%%%%%%%%%%%%%%%%%%%%%%%%%%%%%%%%%%%%%%%%%%%%%%%%%%%%%%%%%%%%%%%%%%%%%%%%%%%%%%%%%%%%%%%%%%%%%%%%%%%%%%%%%%%%%%%%%%%%%%%%%%%%%%%%%%%%%%%%%%%%%%%%%%%%%%%%%%%%%%%%%%%%%%%%%%%%%%%%%%%%%%%%%%%%%%%%%%%%%%%%%%%%%%%%%%%%%%%%%%%%%%%%%%%%%%%%%%%%%%%%%%%%%%%%%%%%%%%%%%
\section{Geometric properties of bounded convex domains} \label{sec:geometry}%%%%%%%%%%%%%%%%%%%%%%%%%%%%%%%%%%%%%%%%%%%%%%%%%%%%%%%%%%%%%%%%%%%%%%%%%%%%%%%%%%%%%%%%%%%%%%%%%%%%%%%%%%%%%%%%%%%%%%%%%%%%%%%%%%%%%%%%%%%%%%%%%%%%%%%%%%%%%%%%%%%%%%%%%%%%%%%%%%%%%%%%%%%%%%%%%%%%%%%%%%%%%%%%%%%%%%%%%%%%%%%%%%%%%%%%%%%%%%%%%%%%%%%%%%%%%%%%%%%%%%%%%%%%%%%%%%%%%%%%%%%%%%%%%%%%%%%%%%%%%%%%%%%%%%%%%%%%%%%%
%%%%%%%%%%%%%%%%%%%%%%%%%%%%%%%%%%%%%%%%%%%%%%%%%%%%%%%%%%%%%%%%%%%%%%%%%%%%%%%%%%%%%%%%%%%%%%%%%%%%%%%%%%%%%%%%%%%%%%%%%%%%%%%%%%%%%%%%%%%%%%%%%%%%%%%%%%%%%%%%%%%%%%%%%%%%%%%%%%%%%%%%%%%%%%%%%%%%%%%%%%%%%%%%%%%%%%%%%%%%%%%%%%%%%%%%%%%%%%%%%%%%%%%%%%%%%%%%%%%%%%%%%%%%%%%%%%%%%%%%%%%%%%%%%%%%%%%%%%%%%%%%%%%%%%%%%%%%%%%%%%%%%%%%%%%%%%%%%%%%%%%%%%%%%%%%%%%%%%%
In this section, we introduce some auxiliary geometric results involving bounded convex domains. 
Throughout this section, $\O \subset \R^3$ denotes a $C^2$ bounded strictly convex domain. 
Our strategy is to properly compare convex domain $\O$ with spherical domains in $\R^3$ 
so that we can build up estimates for convex domains based on those for spherical domains. 

Let $n(q)$ denote the unit outward normal of $\partial\O$ at $q\in\partial\O$, and 
$\hat{v}=\frac{v}{|v|}$ denote the unit vector with the direction $v\in\R^3$. We start with several geometric notations we shall 
frequently use. 
\begin{defi}
    Let $\diam \O$ be the diameter of bounded domain $\O$. Namely, we define 
    \begin{equation}
        \diam \O := \sup_{(x_1,x_2)\in\O\times\O}|x_1-x_2|.
    \end{equation}
    For an interior point $x\in \O$, let 
    \begin{equation} \begin{split}
        d_x :&= \dist(x,\partial\Omega) \\
        &=\inf_{q\in \partial\Omega}|x-q|
    \end{split}\end{equation}
    be the distance from $x$ to $\partial \O$. 
    For $x\in\bar\O$ and $v\in\R^3$, we define $\tau_-(x,v)$ to be the backward exit time for $x$ 
    leaving $\O$ with velocity $-v$, and we define $q_-(x,v)$ to be the corresponding point that 
    the backward trajectory touches $\partial\O$. More precisely, 
    \begin{align*}
        \tau_-(x,v)&:=\inf \{ t>0: \, x -tv \notin \Omega \}, \\
        q_-(x,v)&:=x - \tau_-(x,v)v.
    \end{align*}
    Furthermore, the absolute value of the component of velocity $v$ passing through the surface $\partial\O$ at $q_-(x,v)$ is denoted by
    \begin{equation}
        N_-(x,v):= |n(q_-(x,v))\cdot \hat v|.
    \end{equation}
    In a similar fashion, we can define the corresponding forward concept by
    \begin{align*}
    \tau_+(x,v)&:=\inf \{ t>0: \, x +tv \notin \Omega \}, \\
    q_+(x,v)&:=x + \tau_+(x,v)v, \\
    N_+(x,v)&:= |n(q_+(x,v))\cdot \hat v|.
    \end{align*}
\end{defi}

The following two propositions from \cite{chen19} concern estimates for backward trajectory.
\begin{prop} \label{prop:ProjDistance}
    Let $x$ be an interior point of $\O$ and $v\in\R^3$. Then 
    \[
        |x-q_-(x,v)|\geq \frac{d_x}{N_-(x,v)}.
    \]
\end{prop}
\begin{prop} \label{prop:ProjDistance2}
    Let $x,y$ be interior points of $\O$ and $v\in\R^3$. If \\
    $|x-q_-(x,v)|\leq |y-q_-(y,v)|$, then
    \begin{align}
        &|q_-(x,v)-q_-(y,v)| \leq \frac{|x-y|}{N_-(x,v)}, \\
        &\biggl||x-q_-(x,v)|-|y-q_-(y,v)| \biggr| \leq \frac{2|x-y|}{N_-(x,v)}.
    \end{align}
\end{prop}

\begin{defi} \label{defi:comparison}
    We say a bounded convex domain $\O$ satisfies uniform interior sphere condition (resp. uniform sphere-enclosing condition) 
    if there is a constant $r_1=r_1(\O)>0$ (resp. $R_1=R_1(\O)>0$) such that 
    for any boundary point $q\in\partial\Omega$, there is a sphere $S_i(q)=\partial B_{r_1}(y)$ (resp. $S_o(q)=\partial B_{R_1}(y')$) 
    such that $B_{r_1}(y)\subset\O$ (resp. $\O\subset B_{R_1}(y')$) with $q\in  S_i(q)$ (resp. $q\in S_o(q)$). (See Figure~\ref{fig:uniformsphere}.)
\begin{figure}[!htb]
    \centering
    \includegraphics[scale=0.5]{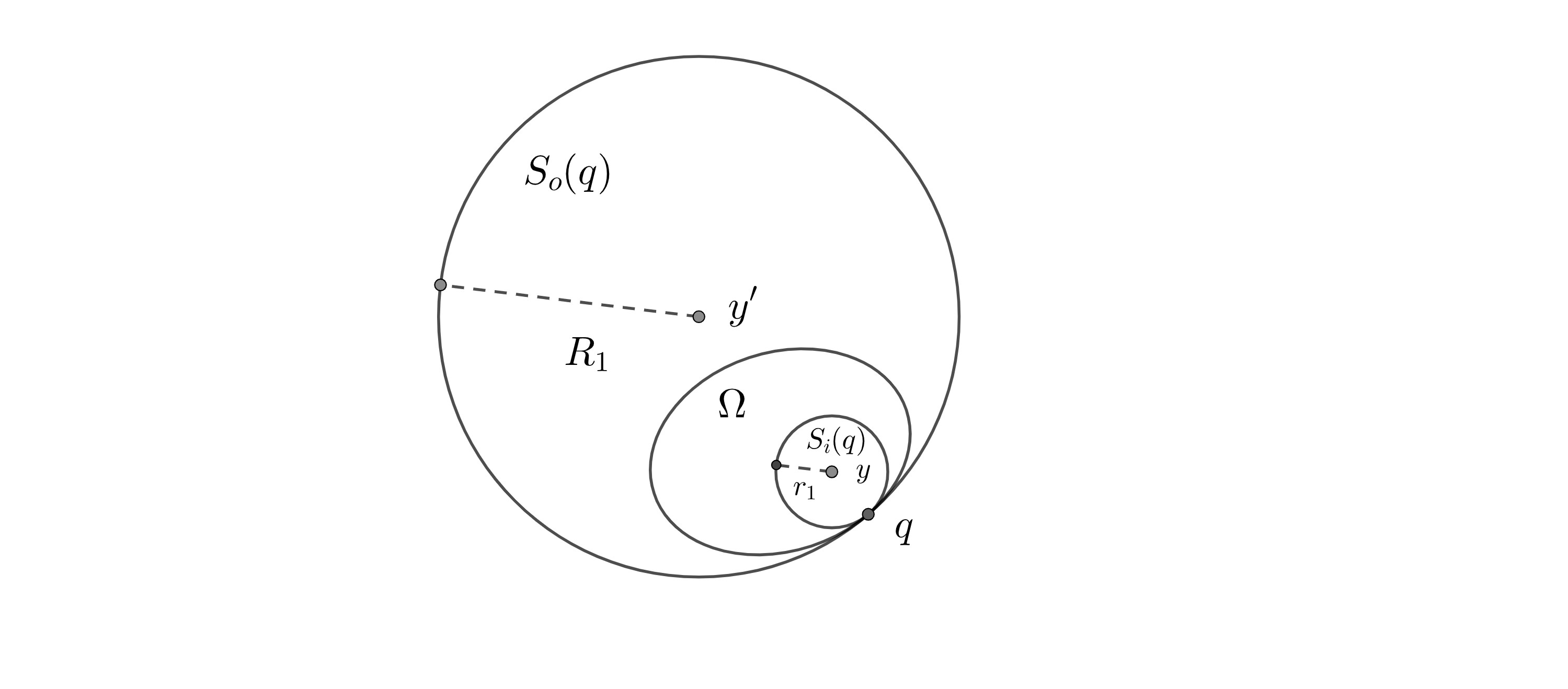}
    \caption{}
    \label{fig:uniformsphere}
\end{figure}
\end{defi}
\begin{prop} \label{prop:curvature}
    Let $\mathcal{D}\subset \R^2$ be a disk of radius $r$ and $A,B$ two points on $\partial\mathcal{D}$. Suppose there is a regular parametrized curve $\alpha(s)$
    connecting $\alpha(s_0)=A$ to $\alpha(s_1)=B$ contained in the 
    circular segment bounded by chord $\overline{AB}$ and minor arc $\stackrel\frown{AB}$. We further assume that $\alpha(s)$ is convex, 
    i.e. the curvature $k(s)$ at $\alpha(s)$ is positive everywhere, $\alpha(s)$ is tangent to $\mathcal{D}$ at $A$, and 
    $\alpha(s)$ leaves $\mathcal{D}$ at $B$ as Figure~\ref{fig:curvatureprop} shows. Then, there exists $s_*\in(s_0,s_1)$ such that 
    $k(s_*)\leq\frac{1}{r}$.
\begin{figure}[!htb]
    \centering
    \includegraphics[scale=0.8]{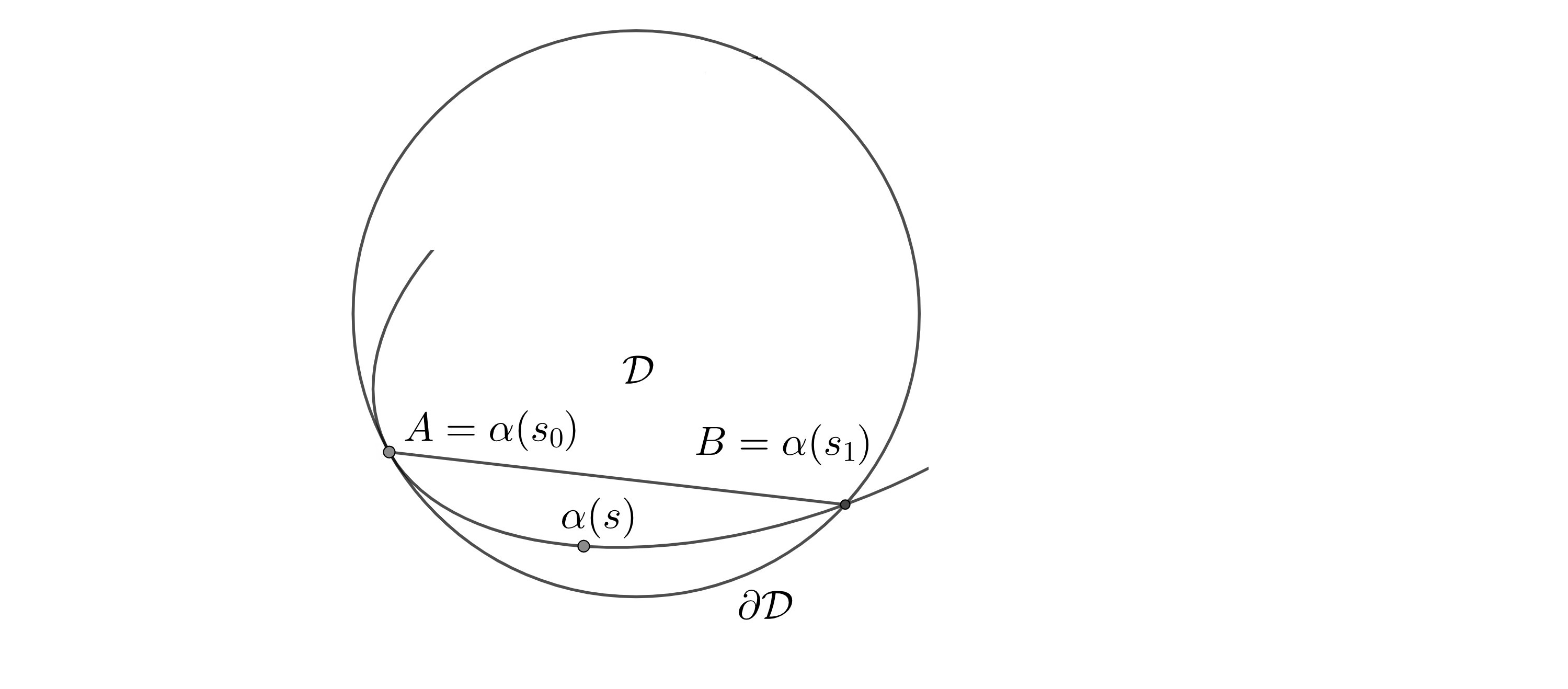}
    \caption{}
    \label{fig:curvatureprop}
\end{figure}
\end{prop}
\begin{proof}
    We may assume $\alpha(s)$ is parametrized by arc length. Let $e_1$ denote the vector $\frac{\overrightarrow{AB}}{|A-B|}$. We choose $e_2\in \S^1$ such that 
    $\{ e_1,e_2 \}$ forms a positively oriented orthonormal basis. 
    Denote the signed angle from vector $e_1$ to tangent vector $\alpha'(s)$ by $\theta(s)$. Therefore, 
    $|A-B|=2r |\sin\theta(s_0)|$. On the other hand, since $\alpha'(s)=\cos\theta(s)e_1+\sin\theta(s)e_2$, we obtain 
    \begin{equation}
        2r |\sin\theta(s_0)|=|A-B|=\int^{s_1}_{s_0} \alpha'(s)\cdot e_1 \,ds=\int^{s_1}_{s_0} \cos\theta(s) \, ds.
    \end{equation}
    Suppose $k(s)> \frac{1}{r}$ for any $s\in(s_0,s_1)$. Then, for any $s\in(s_0,s_1)$,
    \begin{equation}
        \frac{d\theta(s)}{ds}=k(s)> \frac{1}{r},
    \end{equation}
    or
    \begin{equation}
        \frac{ds(\theta)}{d\theta} < r.
    \end{equation}
    Hence, we have 
    \begin{equation}\begin{split}
        2r |\sin\theta(s_0)| =& \int^{s_1}_{s_0} \cos\theta(s) \, ds \\
        =&\int^{\theta_1}_{\theta_0} \cos\theta \frac{ds(\theta)}{d\theta}\,d\theta \\
        <&r\, \int^{\theta_1}_{\theta_0} \cos\theta \,d\theta \\
        =&r\,(\sin\theta_1-\sin\theta_0),
    \end{split}\end{equation}
    where $\theta_0=\theta(s_0)$ and $\theta_1=\theta(s_1)$. Since $\frac{\pi}{2}\geq |\theta_0| \geq\theta_1 >0$, 
    we notice that 
    \begin{equation}
        r\,(\sin\theta_1-\sin\theta_0) \leq 2r |\sin\theta_0|,
    \end{equation}
    which leads to a contradiction.
\end{proof}
\begin{prop} \label{prop:comparison}
    If $\O$ satisfies the positive curvature condition in Definition~\ref{defi:PositiveCurvature}, then it also satisfies both uniform interior sphere condition and uniform sphere-enclosing condition. 
\end{prop}
\begin{proof}
    Since $\partial\O$ is $C^2$ and compact, there is a tubular neighborhood near $\partial\O$. That is, 
    there exists a number $\epsilon >0$ such that whenever $q_1,q_2\in \partial\O$ the segments of the 
    normal lines of length $2\epsilon$, centered at $q_1$ and $q_2$, are disjoint. This fact can be found in \cite{do16}, for example. 
    Let $r_1=\frac{\epsilon}{2}$. For a given point $q\in\partial\O$, therefore we can consider the sphere $B_{r_1}(y)\subset\O$ such that $q\in\partial B_{r_1}(y)$. 
    Clearly, $S_i(q)=\partial B_{r_1}(y)$ satisfies the desired property. 
    
    Concerning uniform sphere-enclosing condition, we shall proceed the proof by contradiction. 
    We denote the principal curvatures of $\partial\O$ at $q$ by $k_1(q)$ and $k_2(q)$ ($k_1(q)\leq k_2(q)$). By positive curvature property of $\partial\O$, 
    we can choose $k_0$ such that $0<k_0<\min_{q\in\partial\O}k_1(q)$. We claim that $R_1= \frac{1}{k_0}$ is a desired radius for the uniform sphere-enclosing condition. 
    For a given point $q\in\partial\O$, there is $y'$ such that $\overrightarrow{y'q} // n(q)$ and $|y'-q|=R_1$. Under suitable rotations and 
    translations, $\partial\O$ can be locally expressed as the graph of a function $z=h(x,y)$ at $q=(0,0,0)$, where 
    \begin{equation} \label{eq:TaylorExpansion}
        h(x,y)=\frac{1}{2}(k_1(q)x^2+k_2(q)y^2)+o(x^2+y^2).
    \end{equation}
    Since $B_{R_1}(y')$ has constant normal curvature $k_0$, $\overline{B_{R_1}(y')}$ contains a neighborhood of $q$ in $\partial\O$ by \eqref{eq:TaylorExpansion}. 
    Suppose $\O\not\subset B_{R_1}(y')$, then there exists $q_1\in\partial\O \cap \partial B_{R_1}(y')$ such that the plane 
    $E$ containing $q,y',q_1$ has the intersection curve with $\partial\O$ satisfies the condition of 
    Proposition~\ref{prop:curvature} with $A=q$, $B=q_1$, and $r=R_1$. Then Proposition~\ref{prop:curvature} implies that 
    there is a point $q_*\in\partial\O$ such that $k(s_*)\leq\frac{1}{R_1}$, where $q_*=\alpha(s_*)$ under the notation from Proposition~\ref{prop:curvature}. 
    Hence, we note that the normal curvature $k_n(q_*)$ of $\partial\O$ at $q_*$ satisfies
    \begin{equation}
        k_n(q_*)\leq k(q_*)\leq \frac{1}{R_1} = k_0,
    \end{equation}
    which is a contradiction. Therefore, we define $S_o(q)=\partial B_{R_1}(y')$. This completes the proof.
\end{proof}
\begin{rmk}
The above proof shows that $\frac{1}{\min\limits_{q\in\partial\O}k_1(q)-\epsilon}$ is a uniform radius for sphere-enclosing condition for 
every small $\epsilon>0$. Letting $\epsilon\to 0^+$, we obtain the smallest and optimal radius $R_1 =\frac{1}{\min\limits_{q\in\partial\O}k_1(q)}$.
\end{rmk}
The following lemma in \cite{chen19} is useful in our estimates.
\begin{lem} \label{lemma:SurfaceIntegral}
    Suppose $\Omega$ satisfies the positive curvature condition in Definition~\ref{defi:PositiveCurvature}. Then there exists a constant $C=C(\O)$ such that, 
    for any interior point $x\in\Omega$, we have 
    \begin{equation}
        \int_{\partial\Omega}\frac 1 {|x-q|^2}\,d\Sigma(q)\leq C(|\log(d_x)|+1),
    \end{equation}
        where $\Sigma(q)$ is the surface element of $\partial\Omega$ at point $q\in\partial\Omega$.
\end{lem}
The following proposition concerns an estimate for chords in a bounded convex domain.
\begin{prop} \label{prop:ChordBound}
For a given bounded convex domain $\O$ satisfying positive curvature condition in Definition~\ref{defi:PositiveCurvature}, there exists a constant $C=C(\O)$ such that 
for any $x\in \bar{\O}$ and $v\in\R^3$, we have
\begin{equation}
    |q_-(x,v)-q_+(x,v)|\leq C N_-(x,v).
\end{equation}
\end{prop}
\begin{proof}
    For given $x\in\bar\O$ and $v\in\R^3$, write $q_-=q_-(x,v)$, 
    $q_+=q_+(x,v)$, and $\theta_- = \arccos N_-(x,v)$ for simplicity. In view of Proposition~\ref{prop:comparison}, there exists 
    sphere $S_o(q_-)$ as defined in Definition~\ref{defi:comparison}. 
    Denote the other intersection of half-line $\overrightarrow{q_-q_+}$ and $S_o(q_-)$ by $q_1$. Therefore, we have 
    \begin{equation}\begin{split}
        |q_--q_+|&\leq |q_--q_1| \\
        &=2R_1 \cos\theta_- \\
        &=2R_1 N_-(x,v).
    \end{split}\end{equation}
\end{proof}
\begin{rmk} \label{remark:comparison}
    From the above proof, we have 
    \begin{equation}
        |q_--q_+|\leq 2R_1 \cos\theta_-.        
    \end{equation}
    Replacing $v$ with $-v$ in Proposition~\ref{prop:ChordBound} yields
    \begin{equation}
        |q_--q_+|\leq 2R_1 \cos\theta_+,
    \end{equation}
     where $\theta_+=\arccos N_+(x,v)$.
\end{rmk}
\begin{prop} \label{prop:DistanceComparison}
    For a given circle $\mathcal C$ in $\R^2$ centered at $O$ with radius $r$ and two given points $A,B$ on $\mathcal C$. Let 
    $N$ be the arc midpoint of  $\stackrel\frown{AB}$ (the minor arc) and $M$ be the midpoint of $\overline{AB}$. 
    Then for any $Y\in \overline{AM}$ (resp. $Y\in \overline{BM}$), we have 
    \begin{equation}
        \dist(Y,\mathcal C)\geq \frac{1}{\sqrt{2}} |Z-Y|,
    \end{equation}
    where $Z$ is the point on $\overline{AN}$ (resp. $\overline{BN}$) such that $(Z-Y)\perp(A-B)$ (See Figure~\ref{fig:DistanceComparison}).
\end{prop}
\begin{proof}
    Without loss of generality, we may assume $Y\in \overline{AM}$. Let $\theta$ denote the angle
    $\angle{OAM}$. Suppose half-line $\overrightarrow{OY}$ meets $\mathcal C,\overline{AN}$ at $P,Q$, respectively. 
    Thus, $|P-Y|=\dist(Y,\mathcal C)$. We also notice that 
    \begin{equation}
        \angle{YZQ}=\angle{ONA}=\frac{\pi}{4}+\frac{\theta}{2}.
    \end{equation}
    Therefore, $\sin{\angle{YZQ}}\geq \frac 1 {\sqrt 2}$.
    By the law of sines, we have 
    \begin{equation}\begin{split}
        |P-Y|\geq |Q-Y|&=\frac{\sin\angle{YZQ}}{\sin\angle{YQZ}}\, |Z-Y| \\
        &\geq \frac{1}{\sqrt{2}}|Z-Y|.
    \end{split}\end{equation}
    \begin{figure}[!htb]
        \centering
        \includegraphics[scale=0.7]{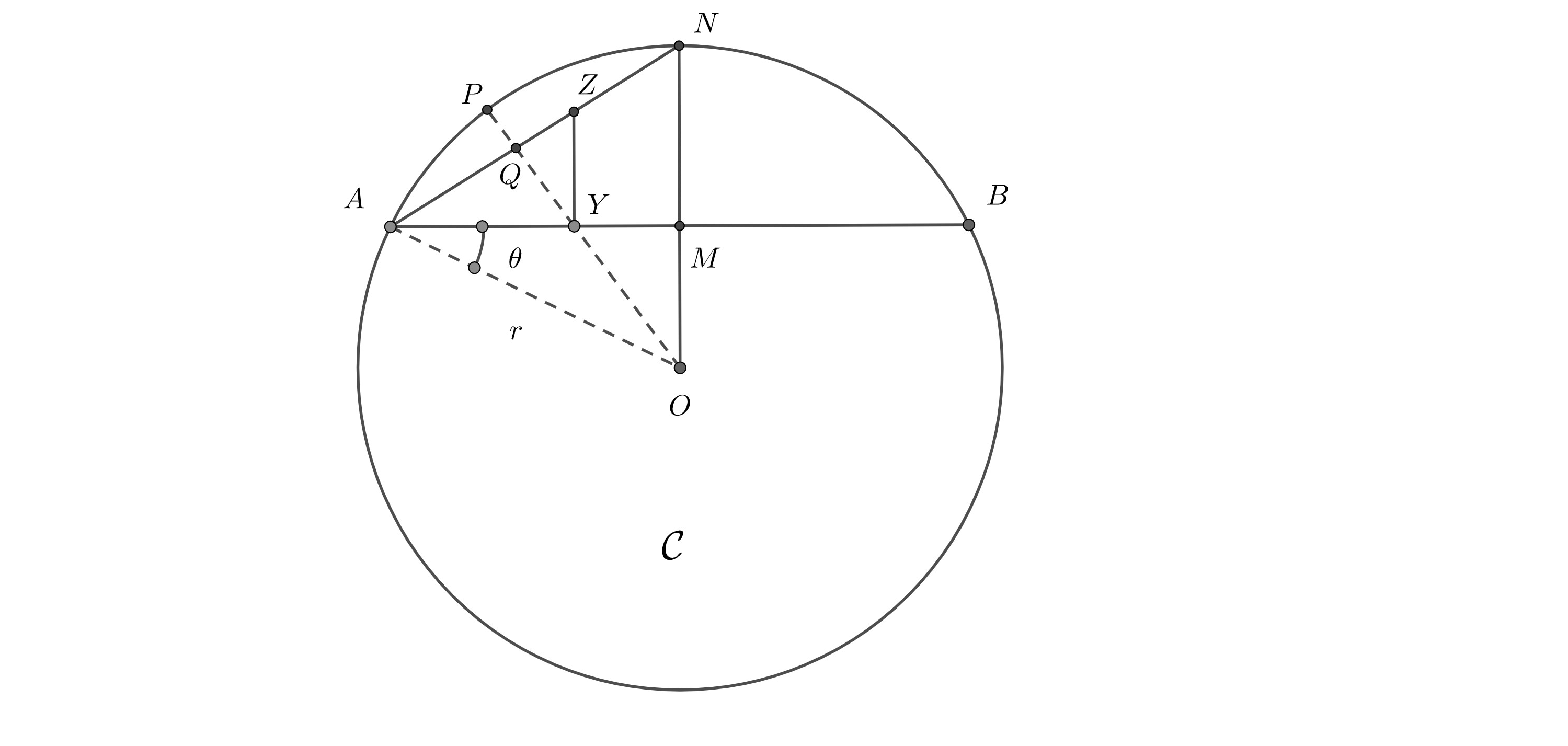}
        \caption{}
        \label{fig:DistanceComparison}
    \end{figure}
\end{proof}
We are now in a position to prove the following lemma.
\begin{lem} \label{lemma:FracIntegralOverChord}
    Suppose $\Omega\subset\R^3$ satisfies the positive curvature condition in Definition~\ref{defi:PositiveCurvature}. Then, 
    there exists a constant $C=C(\O)$ such that for any $y\in\O$ and $\hat{v}\in\S^2$, we have
    \begin{equation} \label{eq:FracIntegralOverChord}
        \int^{|q_+(y,\hat{v})-y|}_0 d_{y+r\hat{v}}^{-\frac 1 2 +\epsilon}\, dr \leq C, \quad \forall\epsilon \in[0,\frac 1 2),
    \end{equation}
    and
    \begin{equation} \label{eq:FracIntegralOverChord2}
        \int^{|q_+(y,\hat{v})-y|}_0 d_{y+r\hat{v}}^{-1 +\epsilon}\, dr \leq \frac{C}{\epsilon} d_y^{-\frac 1 2 +\epsilon}, \quad \forall\epsilon \in(0,\frac 1 2). 
    \end{equation}
\end{lem}
\begin{proof}
    The idea of proof is that we first show the lemma holds for balls. For general cases, we employ Proposition~\ref{prop:comparison} to 
    compare the convex domain $\O$ with balls. For given $\epsilon\in [0,\frac 1 2)$, $y\in\O$ and $\hat{v}\in\S^2$. 
    We denote $q_-=q_-(y,\hat{v})$ and $q_+=q_+(y,\hat{v})$ for simplicity. We 
    claim
    \begin{equation} \label{eq:FracIntegralOverChord1}
        \int^{|q_+-q_-|}_0 d_{q_-+r\hat{v}}^{-\frac 1 2 +\epsilon}\,dr \leq C.
    \end{equation}
    First, we shall prove that \eqref{eq:FracIntegralOverChord1} holds for the case of balls in $\R^3$. 
    Let $\mathcal B\subset \R^3$ be an open ball with radius $\rho$ centered at $O$ and $A,B$ be two points on $\partial\mathcal B$. For 
    $\hat v=\frac{\overrightarrow{AB}}{|A-B|}$, the following integral is bounded by some constant $C=C(\rho)$,
    \begin{equation} \label{eq:FracIntegralOverChordBall}
        \int^{|A-B|}_0 \dist(A+r\hat v,\partial\mathcal B)^{-\frac 1 2 +\epsilon}\,dr \leq C.
    \end{equation}
    Let $\mathcal{C}$ be the intersection circle of $\partial\mathcal{B}$ and the plane passing through $A,B$ and $O$. 
    Denote the midpoint of $\overline{AB}$ by $M$ and write $\theta = \angle{OAB}$. For $0\leq r \leq |A-M|$, it follows by Proposition~\ref{prop:DistanceComparison} that 
    \begin{equation}\begin{split}
        \dist(A+r\hat v,\mathcal C)&\geq \frac{1}{\sqrt{2}} |Z-Y| \\
        &=\frac{1}{\sqrt{2}}\tan\big( \frac{\pi}{4}-\frac{\theta}{2} \big)\,r \\
        &=\frac{1}{\sqrt{2}}\frac{1-\sin\theta}{\cos\theta}\,r,
    \end{split}\end{equation}
    where $Y=A+r\hat v$ and $Z$ is as defined in Proposition~\ref{prop:DistanceComparison}. Therefore, we obtain 
    \begin{equation}\begin{split}
        \int^{|A-M|}_0 \dist(A+r\hat v,\partial\mathcal B)^{-\frac 1 2 +\epsilon}\,dr &= \int^{|A-M|}_0 \dist(A+r\hat v,\mathcal C)^{-\frac 1 2 +\epsilon}\,dr  \\
        &\leq C\int^{\rho\cos\theta}_0\left(\frac{\cos\theta}{1-\sin\theta}\right)^{\frac 1 2 -\epsilon} r^{-\frac 1 2 +\epsilon}\,dr \\
        &\leq C\rho^{\frac 1 2 +\epsilon}\, \frac{\cos\theta}{(1-\sin\theta)^{\frac 1 2 -\epsilon}} \\
        &\leq C\max\{\rho,1\},
    \end{split}\end{equation}
    where the last inequality follows from the fact $\cos\theta \leq \sqrt{2(1-\sin\theta)}$. The situation 
    $|A-M|\leq r\leq|A-B|$ can be treated similarly. Therefore, \eqref{eq:FracIntegralOverChordBall} follows.
    
    For general cases, we make a comparison with sphere cases. 
    According to Proposition~\ref{prop:comparison}, there are spheres $S_i(q_-)$ and $S_i(q_+)$ with radii $r_1=r_1(\O)$ as defined 
    in Definition~\ref{defi:comparison}. Denote the other intersection of $S_i(q_-)$ and 
    $\overline{q_-q_+}$ by $q_1$, the other intersection of $S_i(q_+)$ and $\overline{q_-q_+}$ by $q_2$. For $0\leq r \leq |q_1-q_-|$, we notice that 
    \begin{equation}
        \dist(q_-+r\hat{v},\partial\O)\geq\dist(q_-+r\hat{v},S_i(q_-)).
    \end{equation}
    For $|q_2-q_-|\leq r\leq |q_+-q_-|$, similarly we have 
    \begin{equation}
        \dist(q_-+r\hat{v},\partial\O)\geq\dist(q_-+r\hat{v},S_i(q_+)).
    \end{equation}
    Therefore, in the case of $\overline{q_-q_1}\cup\overline{q_2q_+} = \overline{q_-q_+}$, 
    by \eqref{eq:FracIntegralOverChordBall} there exists a constant $C=C(r_1)$ such that 
    \begin{equation}\begin{split}
        \int^{|q_+-q_-|}_0 d_{q_-+r\hat{v}}^{-\frac 1 2 +\epsilon}\,dr &\leq \int^{|q_1-q_-|}_0 d_{q_-+r\hat{v}}^{-\frac 1 2 +\epsilon}\,dr + \int^{|q_+-q_-|}_{|q_2-q_-|} d_{q_-+r\hat{v}}^{-\frac 1 2 +\epsilon}\,dr \\
        &\leq \int^{|q_1-q_-|}_0 \dist(q_-+r\hat{v},S_i(q_-))^{-\frac 1 2 +\epsilon}\,dr \\
        &\quad \quad+ \int^{|q_+-q_-|}_{|q_2-q_-|} \dist(q_-+r\hat{v},S_i(q_+))^{-\frac 1 2 +\epsilon}\,dr \\
        &\leq C.
    \end{split}\end{equation}
    If $\overline{q_-q_1}\cup\overline{q_2q_+}$ does not cover $\overline{q_-q_+}$, we consider the convex hull of $S_i(q_-)\cup S_i(q_+)$, 
    denoted by $S$. Clearly, one can see that 
    \begin{equation}
        \dist(q_-+r\hat{v},\partial\O)\geq\dist(q_-+r\hat{v},\partial S).
    \end{equation}
    From solid geometry, we can see the following property.
    \begin{claim*}
        For $\left|\frac{q_1+q_-}{2}-q_- \right|\leq r \leq \left|\frac{q_2+q_+}{2}-q_- \right|$, 
        if $N_-(y,\hat{v})\leq N_+(y,\hat{v})$ (resp. $N_-(y,\hat{v})> N_+(y,\hat{v})$) or equivalently $\theta_- \geq \theta_+$ (resp. $\theta_- < \theta_+$), where 
        $\theta_{\pm}= \arccos N_{\pm}(y,\hat{v})$, then we have the inequality 
        \begin{align}
            &\dist(q_-+r\hat{v},\partial S)\geq \dist(\frac{q_1+q_-}{2},S_i(q_-)) = r_1(1-\sin\theta_-) \\
            \bigg(\text{resp. }& \dist(q_-+r\hat{v},\partial S)\geq \dist(\frac{q_2+q_+}{2},S_i(q_+))= r_1(1-\sin\theta_+)\bigg).
        \end{align}
        \begin{figure}[!htb]
            \centering
            \includegraphics[scale=0.7]{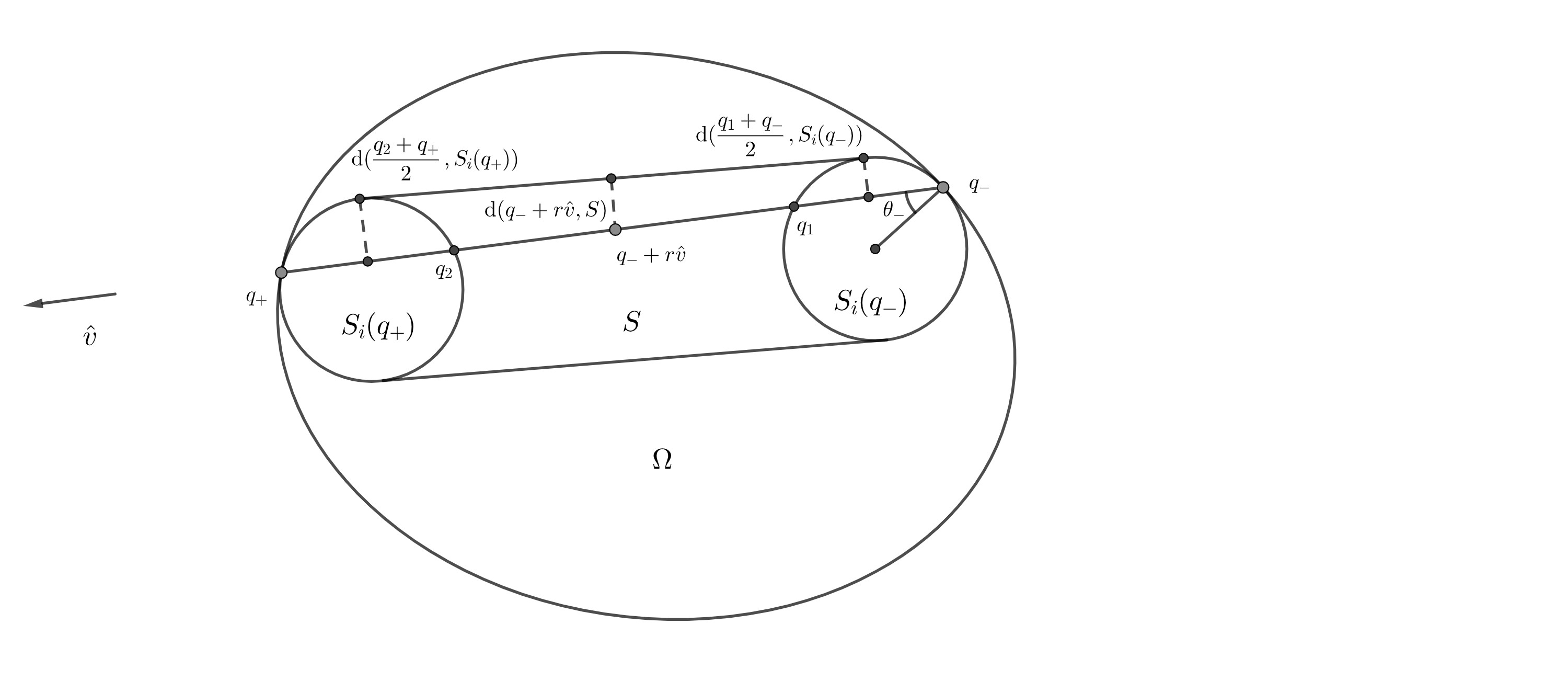}
            \caption{}
            \label{fig:distance}
        \end{figure}
        \noindent See Figure~\ref{fig:distance} for two-dimensional cases.
    \end{claim*}
    \noindent Recalling that $|q_--q_+|\leq 2R_1 \cos\theta_-$ from Remark~\ref{remark:comparison}, thus we deduce that, if $N_-(y,\hat{v})\leq N_+(y,\hat{v})$,
    \begin{equation}\begin{split}
        \int^{|q_+-q_-|}_0 d_{q_-+r\hat{v}}^{-\frac 1 2 +\epsilon}\,dr&\leq \int^{|q_1-q_-|}_0 +\int^{|q_2-q_-|}_{|q_1-q_-|}+\int^{|q_+-q_-|}_{|q_2-q_-|}d_{q_-+r\hat{v}}^{-\frac 1 2 +\epsilon}\,dr\\
        &\leq C(r_1) +2R_1 r_1^{-\frac 1 2 +\epsilon} \cos\theta_-\,(1-\sin\theta_-)^{-\frac 1 2 +\epsilon} +C(r_1) \\
        &\leq C(r_1,R_1).
    \end{split}\end{equation}
    The inequality holds for $N_-(y,\hat{v})> N_+(y,\hat{v})$ similarly. 
    
    Let us now look at \eqref{eq:FracIntegralOverChord2}. We mimic the above proof for \eqref{eq:FracIntegralOverChord1}, since the 
    steps proceed in the same way. For given $\epsilon\in (0,\frac 1 2)$, $y\in\O$ and $\hat{v}\in\S^2$, this time we claim 
    \begin{equation} \label{eq:FracIntegralOverChord21}
        \int^{|q_+-q_-|}_0 d_{q_-+r\hat v}^{-1 +\epsilon}\,dr \leq \frac{C}{\epsilon} d_y^{-\frac 1 2 +\epsilon}
    \end{equation}
    for some constant $C=C(r_1,R_1)$. For the cases of open balls in $\R^3$, we adopt the above notations and deduce that 
    \begin{equation}\begin{split}
        \int^{|A-M|}_0 \dist(A+r\hat{v},\mathcal \partial \mathcal{B})^{-1 +\epsilon}\,dr&=\int^{|A-M|}_0 \dist(A+r\hat{v},\mathcal C)^{-1 +\epsilon}\,dr \\
        &\leq C\int^{\rho\cos\theta}_0\left(\frac{\cos\theta}{1-\sin\theta}\right)^{1 -\epsilon} r^{-1 +\epsilon}\,dr \\
        &\leq C\, \frac{\rho^{\epsilon}\cos\theta}{\epsilon(1-\sin\theta)^{1 -\epsilon}} \\
        &\leq C\,\frac{\max\{\rho,1\}}{\epsilon (1-\sin\theta)^{-\frac 1 2 +\epsilon}} \\
        &\leq\frac{C}{\epsilon}\, \dist(M,\mathcal C)^{-\frac 1 2 +\epsilon} \\
        (&\leq \frac{C}{\epsilon}\, \dist(Y,\mathcal C)^{-\frac 1 2 +\epsilon}\quad \text{for any }Y\in \overline{AB}),
    \end{split}\end{equation}
where we have used $\cos\theta_-\leq \sqrt{2(1-\sin\theta_-)}$. 
For general cases, we denote $q_3=\frac{q_-+q_1}{2}$ and $q_4=\frac{q_++q_2}{2}$ for convenience. By considering the convex hull $S$ 
of $S_i(q_-)\cup S_i(q_+)$ again, if $N_-(y,\hat{v})\leq N_+(y,\hat{v})$, we deduce  
\begin{equation}\begin{split}
    \int^{|q_+-q_-|}_0 d_{q_-+r\hat{v}}^{-1+\epsilon}\,dr&\leq \int^{|q_1-q_-|}_0 +\int^{|q_2-q_-|}_{|q_1-q_-|}+\int^{|q_+-q_-|}_{|q_2-q_-|}d_{q_-+r\hat{v}}^{-1+\epsilon}\,dr\\
    &\leq \frac{C}{\epsilon}\, \dist\bigl((q_3,S_i(q_-)\bigr)^{-\frac 1 2 +\epsilon} \\
    &\quad +2R_1 r_1^{-1+\epsilon} \cos\theta_-\,(1-\sin\theta_-)^{-1+\epsilon}\\
    &\quad +\frac{C}{\epsilon}\,\dist\bigl(q_4,S_i(q_+)\bigr)^{-\frac 1 2 +\epsilon} \\
    &\leq \frac{C}{\epsilon}\, \left( \dist\bigl(q_3,S_i(q_-)\bigr)^{-\frac 1 2 +\epsilon}+ \dist\bigl(q_4,S_i(q_+)\bigr)^{-\frac 1 2 +\epsilon} \right),
\end{split}\end{equation}
where we have used $\cos\theta_- \leq \sqrt{2(1-\sin\theta_-)}$ and 
$ \dist\bigl(q_3,S_i(q_-)\bigr) = r_1(1-\sin\theta_-)$ in the last inequality. 
To complete the proof for \eqref{eq:FracIntegralOverChord21}, we consider sphere $S_o(q_-)$ defined in Definition \ref{defi:comparison}. 
Denote the other intersection of half-line $\overrightarrow{q_-q_+}$ and $S_o(q_-)$ by $q_5$, the midpoint of 
the line segment $\overline{q_-q_5}$ by $q_6$. Clearly, we have
\begin{equation}
    d_y =\dist(y,\partial\O)\leq \dist\bigl(y,S_o(q_-)\bigr)\leq \dist\bigl(q_6,S_o(q_-)\bigr)=2R_1 (1-\sin\theta_-).
\end{equation}
Hence, it follows that 
\begin{equation}\begin{split}
    \dist\bigl(q_3,S_i(q_-)\bigr)^{-\frac 1 2 +\epsilon}&=r_1^{-\frac 1 2 +\epsilon} (1-\sin\theta_-)^{-\frac 1 2 +\epsilon} \\
    &=\left(\frac{r_1}{R_1}\right)^{-\frac 1 2 +\epsilon}\dist\bigl(q_6,S_o(q_-)\bigr)^{-\frac 1 2 +\epsilon} \\
    &\leq C\, d_y^{-\frac 1 2 +\epsilon}.
\end{split}\end{equation}
Similarly, we have 
\begin{equation}
    \dist\bigl(q_4,S_i(q_+)\bigr)^{-\frac 1 2 +\epsilon} \leq C\, d_y^{-\frac 1 2 +\epsilon}.
\end{equation}
We conclude that 
\begin{equation}
    \int^{|q_+-q_-|}_0 d_{q_-+r\hat{v}}^{-1 +\epsilon}\,dr \leq \frac{C}{\epsilon}\, d_y^{-\frac 1 2 +\epsilon}.
\end{equation}
The inequality holds for $N_-(y,\hat{v})> N_+(y,\hat{v})$ similarly. This completes the proof.
\end{proof}
\begin{rmk}
One can see that the behavior of integral $\int^{|q_+-q_-|}_0 d^{-s}_{q_-+r\hat{v}}\,dr$  
depends heavily on 
boundedness of $\Theta(\theta;s)=\frac{\cos \theta}{(1-\sin\theta)^s}$ for $\theta\in (0,\frac{\pi}{2})$. When $s\leq \frac{1}{2}$, there is a constant $C$ independent of $\theta$ 
such that $\Theta(\theta;s) \leq C$. On the other hand, there is no such constant whenever $s>\frac{1}{2}$. We also notice that the 
integral blows up as $s\to 1-$.
\end{rmk}
\begin{cor} \label{cor:DistanceIntegral}
    Suppose $\Omega\subset\R^3$ satisfies the positive curvature condition in Definition~\ref{defi:PositiveCurvature}. Then, there exists a constant $C=C(\O)$ such that, for any small $\epsilon>0$, we have 
    \begin{equation}
        \int_{\O} \frac{1}{d_x^{1-\epsilon}} \, dx \leq \frac{C}{\epsilon}.
    \end{equation}
\end{cor} 
\begin{proof}
Let $y$ be an interior point of $\O$ and $\{ u,v,w\}$ be an orthonormal basis for $\R^3$. We note that 
\begin{equation}\begin{split}
    &\int_{\O} \frac{1}{d_x^{1-\epsilon}} \, dx \\
    =&\int^{|q_+(y,u)-y|}_{-|q_-(y,u)-y|}\int^{|q_+(y+ru,v)-(y+ru)|}_{-|q_-(y+ru,v)-(y+ru)|}\int^{|q_+(y+ru+sv,w)-(y+ru+sv)|}_{-|q_-(y+ru+sv,w)-(y+ru+sv)|}d_{y+ru+sv+tw}^{-1+\epsilon}\, dtdsdr.
\end{split}\end{equation}
According to Lemma~\ref{lemma:FracIntegralOverChord}, it follows that 
\begin{equation}\begin{split}
    &\int^{|q_+(y,u)-y|}_{-|q_-(y,u)-y|}\int^{|q_+(y+ru,v)-(y+ru)|}_{-|q_-(y+ru,v)-(y+ru)|}\int^{|q_+(y+ru+sv,w)-(y+ru+sv)|}_{-|q_-(y+ru+sv,w)-(y+ru+sv)|}d_{y+ru+sv+tw}^{-1+\epsilon}\, dtdsdr \\
    &\,  \leq\int^{|q_+(y,u)-y|}_{-|q_-(y,u)-y|}\int^{|q_+(y+ru,v)-(y+ru)|}_{-|q_-(y+ru,v)-(y+ru)|}  \frac{C}{\epsilon}\,d_{y+ru+sv}^{-\frac{1}{2}+\epsilon}\, dsdr \\
    &\, \leq \frac{C}{\epsilon} \int^{|q_+(y,u)-y|}_{-|q_-(y,u)-y|} C \,dr \\
    &\, \leq \frac{C}{\epsilon}.
\end{split}\end{equation}
\end{proof}
%%%%%%%%%%%%%%%%%%%%%%%%%%%%%%%%%%%%%%%%%%%%%%%%%%%%%%%%%%%%%%%%%%%%%%%%%%%%%%%%%%%%%%%%%%%%%%%%%%%%%%%%%%%%%%%%%%%%%%%%%%%%%%%%%%%%%%%%%%%%%%%%%%%%%%%%%%%%%%%%%%%%%%%%%%%%%%%%%%%%%%%%%%%%%%%%%%%%%%%%%%%%%%%%%%%%%%%%%%%%%%%%%%%%%%%%%%%%%%%%%%%%%%%%%%%%%%%%%%%%%%%%%%%%%%%%%%%%%%%%%%%%%%%%%%%%%%%
\section{Regularity of transport equation in a convex domain} \label{sec:transport}%%%%%%%%%%%%%%%%%%%%%%%%%%%%%%%%%%%%%%%%%%%%%%%%%%%%%%%%%%%%%%%%%%%%%%%%%%%%%%%%%%%%%%%%%%%%%%%%%%%%%%%%%%%%%%%%%%%%%%%%%%%%%%%%%%%%%%%%%%%%%%%%%%%%%%%%%%%%%%%%%%%%%%%%%%%%%%%%%%%%%%%%%%%%%%%%
%%%%%%%%%%%%%%%%%%%%%%%%%%%%%%%%%%%%%%%%%%%%%%%%%%%%%%%%%%%%%%%%%%%%%%%%%%%%%%%%%%%%%%%%%%%%%%%%%%%%%%%%%%%%%%%%%%%%%%%%%%%%%%%%%%%%%%%%%%%%%%%%%%%%%%%%%%%%%%%%%%%%%%%%%%%%%%%%%%%%%%%%%%%%%%%%%%%%%%%%%%%%%%%%%%%%%%%%%%%%%%%%%%%%%%%%%%%%%%%%%%%%%%%%%%%%%%%%%%%%%%%%%%%%%%%%%%%%%%%%%%%%%%%%%%%%%%%
Hereafter, $\O$ denotes a bounded convex domain satisfying positive curvature condition in Definition~\ref{defi:PositiveCurvature} and $g$ denotes an incoming data satisfying Assumption~\ref{assump:incoming}. 
\begin{lem} \label{lemma:jg}
    Let $Jg$ as defined by \eqref{eq:jgdefinition}. Then $Jg\in L^2_v(\R^3;H^{1-\epsilon}_x(\O))$ for any small $\epsilon>0$.
\end{lem}
\begin{proof}
    To prove the lemma, for given $\epsilon>0$ and an incoming data $g$ satisfying Assumption~\ref{assump:incoming}, it suffices to show the boundedness of integral 
    \begin{equation} \label{eq:jgIntegral}
        \int\limits_{\R^3}\int\limits_{\O}\int\limits_{\O} \frac{|Jg(x,v)-Jg(y,v)|^2}{|x-y|^{5-2\epsilon}}\, dxdydv.
    \end{equation}
    We partition the domain of integration $D:=\R^3\times\O\times\O$ into $D_1,D_2$ as below. 
    \begin{align}
        D_1&:=\{(x,y,v)\in D:\, |x-q_-(x,v)|\leq |y-q_-(y,v)| \}, \label{eq:domainD1} \\
        D_2&:=\{(x,y,v)\in D:\, |x-q_-(x,v)|> |y-q_-(y,v)| \}.
    \end{align}
    In view of symmetry of $D_1,D_2$, we only need to calculate the integral \eqref{eq:jgIntegral} over $D_1$. Notice that 
    \begin{equation} \label{eq:jg} 
    \begin{split}
        |Jg(x,v)-Jg(y,v)|^2\leq&2e^{-2\nu(v)\tau_-(x,v)}|g(q_-(x,v),v)-g(q_-(y,v),v)|^2 \\
        &+2|g(q_-(y,v),v)|^2|e^{-\nu(v)\tau_-(x,v)}- e^{-\nu(v)\tau_-(y,v)}|^2. 
    \end{split}
    \end{equation}
    Consequently, we have
    \begin{equation}
        \int_{D_1} \frac{|Jg(x,v)-Jg(y,v)|^2}{|x-y|^{5-2\epsilon}}\, dxdydv \leq I_1 + I_2,
    \end{equation}
    where
    \begin{align}
        I_1 &:= \int_{D_1} \frac{2e^{-2\nu(v)\tau_-(x,v)}|g(q_-(x,v),v)-g(q_-(y,v),v)|^2 }{|x-y|^{5-2\epsilon}}\, dxdydv, \\
        I_2 &:= \int_{D_1} \frac{2|g(q_-(y,v),v)|^2|e^{-\nu(v)\tau_-(x,v)}- e^{-\nu(v)\tau_-(y,v)}|^2}{|x-y|^{5-2\epsilon}}\, dxdydv.
    \end{align}
    To estimate $I_1$, by H\"{o}lder continuity \eqref{eq:condition3} and the 
    fact from condition \eqref{eq:condition2} that
    \begin{equation} \label{eq:jgI1estimate1}
        |g(q_-(x,v),v)-g(q_-(y,v),v)|\leq C e^{-a|v|^2},
    \end{equation}
    we obtain
    \begin{equation}\begin{split} \label{eq:jgI1estimate2}
        &e^{-2\nu(v)\tau_-(x,v)}|g(q_-(x,v),v)-g(q_-(y,v),v)|^{2} \\
        &\qquad= e^{-2\nu(v)\tau_-(x,v)}|g(q_-(x,v),v)-g(q_-(y,v),v)|^{2-\epsilon} \\
        &\qquad\quad \times |g(q_-(x,v),v)-g(q_-(y,v),v)|^{\epsilon} \\
        &\qquad\leq C\, e^{-2\nu(v)\tau_-(x,v)}|q_-(x,v)-q_-(y,v)|^{2-\epsilon} e^{-\epsilon a |v|^2}.
    \end{split}\end{equation}
    According to Proposition~\ref{prop:ProjDistance2}, we have 
    \begin{equation} \label{eq:jgI1estimate3}
        |q_-(x,v)-q_-(y,v)|\leq \frac 1 {N_-(x,v)}|x-y|,
    \end{equation}
    whenever $|x-q_-(x,v)|\leq |y-q_-(y,v)|$. Proposition~\ref{prop:ProjDistance} implies that
    \begin{equation}\begin{split} \label{eq:jgI1estimate4}
        e^{-2\nu(v)\tau_-(x,v)} \leq& \frac{1}{\nu(v)^{1-\epsilon}{\tau_-(x,v)^{1-\epsilon}}} \\
        \leq& \frac{1}{\nu(v)^{1-\epsilon}{\tau_-(x,v)^{1-\epsilon}}}\cdot \left( \frac{N_-(x,v)|x-q_-(x,v)|}{d_x} \right)^{1-\epsilon} \\
        \leq& \frac{N_-(x,v)^{1-\epsilon}|v|^{1-\epsilon}}{\nu(v)^{1-\epsilon}d_x^{1-\epsilon}}.
    \end{split}\end{equation}
    Combining \eqref{eq:jgI1estimate1},\eqref{eq:jgI1estimate2},\eqref{eq:jgI1estimate3}, and \eqref{eq:jgI1estimate4}, we have
    \begin{equation}\begin{split} \label{eq:jgIntegral1}
        I_1&\leq C\, \int\limits_{D_1} \frac{|v|^{1-\epsilon}e^{-\epsilon a|v|^2}}{\nu(v)^{1-\epsilon}d_x^{1-\epsilon}N_-(x,v)|x-y|^{3-\epsilon}}\, dxdydv \\
        & \leq C\, \int\limits_{\O} \int\limits_{\R^3} \int\limits_{\O} \frac{|v|^{1-\epsilon}e^{-\epsilon a|v|^2}}{\nu(v)^{1-\epsilon}d_x^{1-\epsilon}N_-(x,v)|x-y|^{3-\epsilon}}\, dydvdx \\
        & \leq C\, \int\limits_{\O} \int\limits_{\R^3} \frac{|v|^{1-\epsilon}e^{-\epsilon a|v|^2}}{d_x^{1-\epsilon}{N_-(x,v)}}\, dvdx. 
    \end{split}\end{equation}
    For fixed $x\in\O$, we introduce a change of variable $v=(x-z)l$ with $z\in\partial\O$ and $0\leq l<\infty$. For any local chart of $\partial\O$, 
    say $z=\phi(\alpha,\beta)$, we note that 
    \[
    v(\alpha,\beta,l)=\bigl(x-\phi(\alpha,\beta)\bigr)l
    \]
    and the Jacobian for $v$ is given by 
    \begin{equation}\begin{split}
        \left| \det \mathbf J_{v}(\alpha,\beta,l) \right|&=\left|  \partial_l v\cdot (\partial_\alpha v \times \partial_\beta v) \right| \\
        &=l^2\left| (x-\phi(\alpha,\beta))\cdot (\phi_\alpha \times \phi_\beta) \right| \\
        &=l^2\left| (x-\phi(\alpha,\beta))\cdot n(\phi(\alpha,\beta))  \right| |\phi_\alpha \times \phi_\beta|.
    \end{split}\end{equation}
    Covering $\partial\O$ by finitely many such coordinate charts, we obtain the formula for change of variables
    \begin{equation} \label{eq:jgChangeOfVariable}
        \int_{\R^3}h(v)\, dv= \int_{\partial\O}\int^\infty_0 h((x-z)l)\, l^2\left| (x-z)\cdot n(z) \right| \, dld\Sigma(z),
    \end{equation}
    for any $h\in L^1(\R^3)$. Therefore, by the change of variables, we obtain 
    \begin{equation} \begin{split} \label{eq:jgIntegral1ChangeOfVariable1}
        &\int\limits_{\O} \int\limits_{\R^3} \frac{|v|^{1-\epsilon}e^{-\epsilon a|v|^2}}{d_x^{1-\epsilon}{N_-(x,v)}}\, dvdx   \\
        &\qquad= \int\limits_{\O} \int\limits_{\partial\O}\int\limits^\infty_0 \frac{ e^{-\epsilon a|x-z|^2l^2}|x-z|^{1-\epsilon}l^{1-\epsilon}}{d_x^{1-\epsilon}\cdot \frac{|(x-z)\cdot n(z)|}{|x-z|}}\, l^2|(x-z)\cdot n(z)|dld\Sigma(z)dx \\
        &\qquad= \int\limits_{\O} \int\limits_{\partial\O}\int\limits^\infty_0  \frac{e^{-\epsilon a|x-z|^2l^2}|x-z|^{2-\epsilon}l^{3-\epsilon}}{d_x^{1-\epsilon}}\, dld\Sigma(z)dx. 
    \end{split}\end{equation}
    Letting $s=|x-z| \, l$ yields 
    \begin{equation}\begin{split} \label{eq:jgIntegral1ChangeOfVariable2}
        I_1&\leq C\,\int\limits_{\O} \int\limits_{\partial\O}\int\limits^\infty_0  \frac{e^{-\epsilon a|x-z|^2l^2}|x-z|^{2-\epsilon}l^{3-\epsilon}}{d_x^{1-\epsilon}}\, dld\Sigma(z)dx \\
        &= C\, \int\limits_{\O} \int\limits_{\partial\O}\int\limits^\infty_0 \frac{e^{-\epsilon as^2}s^{3-\epsilon}}{d_x^{1-\epsilon}|x-z|^2}\, dsd\Sigma(z)dx  \\
        &\leq C\,\int\limits_{\O} \int\limits_{\partial\O} \frac{1}{d_x^{1-\epsilon}|x-z|^2}\, d\Sigma(z)dx  \\
        &\leq C\, \int_{\O} \frac {|\log(d_x)|+1}{d_x^{1-\epsilon}} \,dx \\
        &\leq C.
    \end{split}\end{equation}
    The third inequality above follows from Lemma~\ref{lemma:SurfaceIntegral} and the last inequality follows from 
    Corollary~\ref{cor:DistanceIntegral} and the fact 
    \begin{equation}
        \frac {|\log(d_x)|}{d_x^{1-\epsilon}} \leq \frac{1}{d_x^{1-\frac{\epsilon}{2}}}
    \end{equation}
    for small $d_x>0$.

    Concerning $I_2$, the condition \eqref{eq:condition2} implies that
    \begin{equation}\begin{split} \label{eq:jgI2estimate1}
        &|g(q_-(y,v),v)|^2|e^{-\nu(v)\tau_-(x,v)}- e^{-\nu(v)\tau_-(y,v)}|^2 \\
        &\qquad \leq C e^{-2a|v|^2}|e^{-\nu(v)\tau_-(x,v)}- e^{-\nu(v)\tau_-(y,v)}|^{2-\epsilon}.
    \end{split}\end{equation}
    By the mean value theorem and Proposition~\ref{prop:ProjDistance2}, we obtain
    \begin{equation}\begin{split} \label{eq:jgI2estimate2}
        |e^{-\nu(v)\tau_-(x,v)}- e^{-\nu(v)\tau_-(y,v)}|&\leq \nu(v)e^{-\nu(v)\tau_-(x,v)}|\tau_-(x,v)-\tau_-(y,v)| \\
        &\leq\frac{2\nu(v)e^{-\nu(v)\tau_-(x,v)}|x-y|}{N_-(x,v)|v|}.
    \end{split}\end{equation}
    Proposition~\ref{prop:ProjDistance} implies that 
    \begin{equation}\begin{split} \label{eq:jgI2estimate3}
        e^{-(2-\epsilon)\nu(v)\tau_-(x,v)}\leq& \frac 1 {\nu(v)^{1-\epsilon}\tau_-(x,v)^{1-\epsilon}} \\
        \leq& \frac{N_-(x,v)^{1-\epsilon}|v|^{1-\epsilon}}{\nu(v)^{1-\epsilon}d_x^{1-\epsilon}}.
    \end{split}\end{equation}
    Taking \eqref{eq:jgI2estimate1},\eqref{eq:jgI2estimate2}, and \eqref{eq:jgI2estimate3} into consideration, we obtain
    \begin{equation} \label{eq:jgIntegral2}
       I_2  \leq C \, \int_{D_1} \frac{e^{-2a|v|^2}\nu(v)\,dxdydv}{|x-y|^{3-\epsilon}N_-(x,v)|v|d_x^{1-\epsilon}}.
    \end{equation}
    Comparing \eqref{eq:jgIntegral2} with \eqref{eq:jgIntegral1}, one can repeat the steps in \eqref{eq:jgIntegral1}, \eqref{eq:jgIntegral1ChangeOfVariable1}, and \eqref{eq:jgIntegral1ChangeOfVariable2} to obtain the following boundedness,
    \[
        I_2 \leq C.
    \]
    This completes the proof. 
\end{proof}
Since $K:L^2(\R^2_v)\to L^2(\R^2_v)$ is a bounded operator regarding velocity variable, we can see that $K$ would preserve regularity in space variable. 
\begin{prop}
    The operator $K: L^2_v(\R^3;H^{1-\epsilon}_x(\O))\to L^2_v(\R^3;H^{1-\epsilon}_x(\O))$ is bounded for any small $\epsilon>0$.
\end{prop}
Next, we deal with regularity preservation of $\SO$. Recall 
\begin{equation}
    \SO h(x,v)=\int^{\tau_-(x,v)}_0 e^{-\nu(v)s}h(x-sv,v)\,ds.
\end{equation}
\begin{lem} \label{lemma:SOPreservation}
    Suppose $h\in L^2_v(\R^3;H^{1-\epsilon}_x(\O))$ and there exist positive constants $a$ and $C$ such that 
    $|h(x,v)|\leq C e^{-a|v|^2}$. Then $\SO h\in L^2_v(\R^3;H^{1-\epsilon}_x(\O))$ for any small $\epsilon>0$.
\end{lem}
\begin{proof}
    To prove the lemma, for given $\epsilon>0$ and $h\in L^2_v(\R^3;H^{1-\epsilon}_x(\O))$ with $|h(x,v)|\leq C e^{-a|v|^2}$, 
    it suffices to show the boundedness of integral 
    \begin{equation}
        \int\limits_{D_1} \frac{|\SO h(x,v)-\SO h(y,v)|^2}{|x-y|^{5-2\epsilon}}\, dxdydv,
    \end{equation}
    where $D_1$ as defined by \eqref{eq:domainD1}. In domain $D_1$, we have 
    \begin{equation}\begin{split}\label{eq:somega}
        |S_\O h (x,v)-S_\O h(y,v)|^2 \leq& 2 \left| \int^{\tau_-(x,v)}_0 e^{-\nu(v)s}(h(x-sv,v)-h(y-sv,v))\,ds \right|^2  \\
        &+2 \left| \int^{\tau_-(y,v)}_{\tau_-(x,v)}e^{-\nu(v)s}h(y-sv,v)\,ds \right|^2. 
    \end{split}\end{equation}
    Therefore, we have 
    \begin{equation}
        \int\limits_{D_1} \frac{|\SO h(x,v)-\SO h(y,v)|^2}{|x-y|^{5-2\epsilon}}\, dxdydv \leq I_1 +I_2,
    \end{equation}
    where
    \begin{align}
        I_1 &:= \int_{D_1} \frac{2 \left| \int^{\tau_-(x,v)}_0 e^{-\nu(v)s}(h(x-sv,v)-h(y-sv,v))\,ds \right|^2 }{|x-y|^{5-2\epsilon}}\, dxdydv, \\
        I_2 &:= \int_{D_1} \frac{2 \left| \int^{\tau_-(y,v)}_{\tau_-(x,v)}e^{-\nu(v)s}h(y-sv,v)\,ds \right|^2}{|x-y|^{5-2\epsilon}}\, dxdydv.
    \end{align}
    Concerning $I_1$, by Cauchy-Schwarz inequality, we obtain 
    \begin{equation}\begin{split} \label{eq:SOextension1}
        &\left| \int^{\tau_-(x,v)}_0 e^{-\nu(v)s}(h(x-sv,v)-h(y-sv,v))\,ds \right|^2 \\
        &\quad \leq\left( \int^{\infty}_0 e^{-\nu(v)s}\,ds\right)\left(\int^{\tau_-(x,v)}_0e^{-\nu(v)s}\bigl(h(x-sv,v)-h(y-sv,v) \bigr)^2 \,ds \right) \\
        &\quad \leq \frac 1 {\nu_0} \int^{\tau_-(x,v)}_0e^{-\nu(v)s}\bigl(h(x-sv,v)-h(y-sv,v) \bigr)^2 \,ds \\
        &\quad\leq  \frac 1 {\nu_0} \int^{\infty}_0e^{-\nu(v)s}\bigl(\bar{h}(x-sv,v)-\bar{h}(y-sv,v) \bigr)^2 \,ds,
    \end{split}\end{equation}
    where $\bar{h}\in L^2_v(\R^3;H^{1-\epsilon}_x(\R^3))$ is an extension of $h\in L^2_v(\R^3;H^{1-\epsilon}_x(\O))$ as defined in \cite[Theorem~5.4]{di11}. 
    \begin{rmk}
    By Theorem~5.4 from \cite{di11}, 
    we know $H^s(\O)$ can be continuously embedded in $H^s(\R^3)$. That is, 
    there exists a constant $C=C(s,\O)$ such that, for given $u\in H^s(\O)$, there is an extension $\overline{u}\in H^s(\R^3)$ of $u$, i.e., $\left.\overline{u}\right|_{\O}=u$, satisfying 
    \[
        \| \overline{u} \|_{ H^s(\R^3)} \leq C\, \| u \|_{ H^s(\O)}.
    \]
    For a.e. fixed $v\in\R^3$, we then define $\bar{h}(x,v)=\overline{h(\cdot,v)}(x)\in L^2_v(\R^3;H^{1-\epsilon}_x(\R^3))$. 
    \end{rmk}
    \noindent Consequently, we deduce that 
    \begin{equation}\begin{split} \label{eq:SOextension}
        I_1&\leq C\,\int\limits_{\R^3}\int\limits_{\R^3}\int\limits_{\R^3} \int^{\infty}_0\frac{e^{-\nu(v)s}\bigl(\bar{h}(x-sv,v)-\bar{h}(y-sv,v)\bigr)^2}{|x-y|^{5-2\epsilon}}  \,dsdxdydv \\
        &\leq C\,\int\limits_{\R^3}\int\limits_{\R^3}\int\limits_{\R^3} \int^{\infty}_0\frac{e^{-\nu(v)s'}\bigl(\bar{h}(x',v')-\bar{h}(y',v')\bigr)^2}{|x'-y'|^{5-2\epsilon}}  \,ds' dx'dy'dv' \\
        &\leq C \,\| \bar{h} \|^2_{ L^2_v(\R^3;H^{1-\epsilon}_x(\R^3))}  \\
        &\leq C\, \| h \|^2_{ L^2_v(\R^3;H^{1-\epsilon}_x(\O))}, 
    \end{split}\end{equation}
    where we have used the change of variables
    \begin{equation} 
        \left\{ 
            \begin{array}{l}
                v'=v, \\
                y'=y-sv, \\
                x'=x-sv, \\
                s'=s.  
            \end{array} 
            \right.
        \end{equation}

    Regarding $I_2$, we notice that 
    \begin{equation} \label{eq:SOPreservationEstimate}
        \begin{split}
        &\left| \int^{\tau_-(y,v)}_{\tau_-(x,v)}e^{-\nu(v)s}h(y-sv,v)\,ds \right|^2 \\
        &\qquad\leq C\,\left| \int^{\tau_-(y,v)}_{\tau_-(x,v)}e^{-\nu(v)s}e^{-a|v|^2}\,ds \right|^2 \\
        &\qquad = C\, e^{-2a|v|^2}\frac{1}{\nu(v)^2}|e^{-\nu(v)\tau_-(x,v)}-e^{-\nu(v)\tau_-(y,v)}|^2  \\
        &\qquad\leq C\, e^{-2a|v|^2}\frac{1}{\nu(v)^2}|e^{-\nu(v)\tau_-(x,v)}-e^{-\nu(v)\tau_-(y,v)}|^{2-\epsilon}  \\
        &\qquad\leq C\, e^{-2a|v|^2}\frac{1}{\nu(v)^\epsilon}\, e^{-(2-\epsilon)\nu(v)\tau_-(x,v)}|\tau_-(x,v)-\tau_-(y,v)|^{2-\epsilon}. 
    \end{split}\end{equation}
    Proposition~\ref{prop:ProjDistance} implies that 
    \begin{equation} \label{eq:SOPreservationEstimate1}
    \begin{split}
        e^{-(2-\epsilon)\nu(v)\tau_-(x,v)}\leq&C\,\frac{1}{\nu(v)^{1-\epsilon}\tau_-(x,v)^{1-\epsilon}}\\
        \leq&C\,  \frac{N_-(x,v)^{1-\epsilon}|v|^{1-\epsilon}}{\nu(v)^{1-\epsilon}d_x^{1-\epsilon}}.
    \end{split}\end{equation}
    By Proposition~\ref{prop:ProjDistance2}, we have 
    \begin{equation} \label{eq:SOPreservationEstimate2}
        |\tau_-(x,v)-\tau_-(y,v)|^{2-\epsilon} \leq C\, \frac{|x-y|^{2-\epsilon}}{N_-(x,v)^{2-\epsilon}|v|^{2-\epsilon}}.    
    \end{equation}
    Combining \eqref{eq:SOPreservationEstimate}-\eqref{eq:SOPreservationEstimate2}, we deduce 
    \begin{equation} \label{eq:SOPintegral}
        I_2 \leq C\, \int\limits_{D_1} \frac{e^{-2a|v|^2}}{N_-(x,v)|x-y|^{3-\epsilon}|v|d_x^{1-\epsilon}} \,dxdydv .
    \end{equation}
    To show the boundedness of above integral on the right, 
    comparing \eqref{eq:SOPintegral} with \eqref{eq:jgIntegral1}, one can repeat the steps in \eqref{eq:jgIntegral1}, \eqref{eq:jgIntegral1ChangeOfVariable1}, and \eqref{eq:jgIntegral1ChangeOfVariable2} to 
    conclude $\SO h\in L^2_v(\R^3;H^{1-\epsilon}_x(\O))$.
\end{proof}
\begin{cor}
    Let $g_i$ be as defined by \eqref{eq:giDef}. Then, \\
    $g_i \in  L^2_v(\R^3;H^{1-\epsilon}_x(\O))$ for each $i\geq 0$.   
\end{cor}
\begin{proof}
    We prove $g_i \in  L^2_v(\R^3;H^{1-\epsilon}_x(\O))$ by induction on $i\geq 0$. \\
    \underline{\textit{Step 1} }.\\
    For $i=0$, since $g_0=Jg$, $g_0\in  L^2_v(\R^3;H^{1-\epsilon}_x(\O)) $ follows by Lemma~\ref{lemma:jg}. For $i=1$, we notice that 
    \begin{equation}
        |Jg(x,v)| \leq |g(q_-(x,v),v)| \leq Ce^{-a|v|^2}.    
    \end{equation}
    In view of Lemma~\ref{lemma:MaxwellianBound}, we have  
    \begin{equation}
        |KJg(x,v)| \leq  Ce^{-a|v|^2},
    \end{equation}
    where we may assume $a<\frac 1 4$. With this in mind,  
    Lemma~\ref{lemma:SOPreservation} guarantees that $g_1 \in L^2_v(\R^3;H^{1-\epsilon}_x(\O))$. 
    Moreover, by noticing that 
    \begin{equation} \label{eq:SOGaussian}
        \SO (e^{-a|v|^2})(x,v)=\int^{\tau_-(x,v)}_0  e^{-\nu(v)s}e^{-a|v|^2} \,ds \leq \frac 1 {\nu_0}  e^{-a|v|^2},
    \end{equation}
    we see that 
    \begin{equation}
        |g_1(x,v)|\leq Ce^{-a|v|^2}.
    \end{equation} 
    \underline{\textit{Step 2}}. \\
    Suppose for some $i\geq 2$, we have 
    \begin{equation}
    g_{i-1}=(\SO K)^{i-1} Jg \in  L^2_v(\R^3;H^{1-\epsilon}_x(\O))
    \end{equation} 
    and 
    \begin{equation}
        |g_{i-1}(x,v)|\leq Ce^{-a|v|^2}. 
    \end{equation}
    Then Lemma~\ref{lemma:MaxwellianBound} implies 
    \begin{equation}
    |Kg_{i-1}(x,v)| \leq  Ce^{-a|v|^2}.
    \end{equation} 
    Applying Lemma~\ref{lemma:SOPreservation} again yields 
    \begin{equation}
        g_i \in  L^2_v(\R^3;H^{1-\epsilon}_x(\O)).
    \end{equation} 
    Finally, we also notice that \eqref{eq:SOGaussian} implies
    \begin{equation}
        |g_{i}(x,v)|\leq Ce^{-a|v|^2}.
    \end{equation} \\
    We conclude that $g_i \in  L^2_v(\R^3;H^{1-\epsilon}_x(\O))$ for each $i\geq 0$ by induction.
\end{proof}
%%%%%%%%%%%%%%%%%%%%%%%%%%%%%%%%%%%%%%%%%%%%%%%%%%%%%%%%%%%%%%%%%%%%%%%%%%%%%%%%%%%%%%%%%%%%%%%%%%%%%%%%%%%%%%%%%%%%%%%%%%%%%%%%%%%%%%%%%%%%%%%%%%%%%%%%%%%%%%%%%%%%%%%%%%%%%%%%%%%%%%%%%%%%%%%%%%%%%%%%%%%%%%%%%%%%%%%%%%%%%%%%%%%%%%%
\section{Regularity via velocity averaging} \label{sec:averaging}%%%%%%%%%%%%%%%%%%%%%%%%%%%%%%%%%%%%%%%%%%%%%%%%%%%%%%%%%%%%%%%%%%%%%%%%%%%%%%%%%%%%%%%%%%%%%%%%%%%%%%%%%%%%%%%%%%%%%%%%%%%%%%%%%%%%%%%%%%%%%%%%%%%%%%%%%%%%%%%%%%%%%%
%%%%%%%%%%%%%%%%%%%%%%%%%%%%%%%%%%%%%%%%%%%%%%%%%%%%%%%%%%%%%%%%%%%%%%%%%%%%%%%%%%%%%%%%%%%%%%%%%%%%%%%%%%%%%%%%%%%%%%%%%%%%%%%%%%%%%%%%%%%%%%%%%%%%%%%%%%%%%%%%%%%%%%%%%%%%%%%%%%%%%%%%%%%%%%%%%%%%%%%%%%%%%%%%%%%%%%%%%%%%%%%%%%%%%%%
This section is devoted to the regularity due to velocity averaging. In contrast to Section~\ref{sec:VelAvgLBE}, we address the 
velocity averaging effect in bounded convex domains instead of the whole space. As mentioned in the introduction, the method of Fourier 
transform does not apply to bounded domains. To remedy this crux, we adopt Slobodeckij semi-norm as an alternative concept of Sobolev function 
class, see Definition~\ref{FracSobolev}. Hence, the difficulty shifts to estimates of singular integrals. To carry out this 
strategy, we introduce changes of coordinates, see Lemma~\ref{lemma:ChangeOfVariable1} and Lemma~\ref{lemma:ChangeOfVariable2}, to 
obtain the boundedness of the aforementioned singular integrals.

We begin with the proof of Corollary~\ref{cor:ksok}. Recall from \eqref{eq:ZeroExtension} that the notation $\widetilde{h}$ denotes the zero extension of $h$ from $\O\times\R^3$ to 
$\R^3\times\R^3$ and $Z:\O\times\R^3\to \R^3\times\R^3$ is the zero extension operator from $\O\times\R^3$ to $\R^3\times\R^3$. 
\begin{proof}[Proof of Corollary~\ref{cor:ksok}]
    For given $f\in L^2(\O\times\R^3)$, according to Lemma~\ref{lemma:ksk}, we have 
    $KSK \widetilde{f}\in L^2_v(\R^3;\tilde{H}^{\sfrac 1 2}_x(\R^3))$. Noticing \eqref{eq:SobolevEquivalence} and 
    \begin{equation}
        \biggl. \bigl(KSK \widetilde{f}\bigr)\biggr|_{\O\times\R^3} = K\SO K f,
    \end{equation}
    we see that 
    \begin{equation}\begin{split}
        \|  K\SO K f \|_{L^2_v(\R^3;H^{\sfrac 1 2}_x(\O))}&\leq  \left\|   \bigl(KSK \widetilde{f}\bigr) \right\|_{L^2_v(\R^3;H^{\sfrac 1 2}_x(\R^3))} \\
        &\leq C\, \left\|   \bigl(KSK \widetilde{f}\bigr) \right\|_{L^2_v(\R^3;\tilde{H}^{\sfrac 1 2}_x(\R^3))} \\
        &\leq C\,\|  \widetilde{f}\|_{L^2(\R^3\times\R^3)} \\
        &=C\,\| f \|_{L^2(\O\times\R^3)}.
    \end{split}\end{equation}
\end{proof}
\begin{lem} \label{lemma:ChangeOfVariable1}
We have the following formula for change of variables for any non-negative measurable function $h=h(v,y,r)$.
\begin{equation}\begin{split}\label{eq:ChangeOfVariable1}
    &\int_{\R^3}\int_{\O}\int^{|y-q_-(y,v)|}_0 h(v,y,r) \, drdydv  \\
    &\qquad = \int_{\R^3}\int_{\O} \int^{|y'-q_+(y',v')|}_0 h(v',y'+r'\hat{v'},r') \, dr'dy'dv'. 
\end{split}\end{equation}
\end{lem}
\begin{proof}
    Let $A$ and $B$ denote the domain of integration on the left hand side and right hand side of \eqref{eq:ChangeOfVariable1}, respectively. 
    We consider the change of variables 
    \begin{equation} 
    \left\{ 
        \begin{array}{l}
            v'=v, \\
            y'=y-r\hat{v}, \\
            r'=r.  
        \end{array} 
        \right.
    \end{equation}
    Let $X:A\to B$ be defined by 
    $X(v,y,r)=(v,y-r\hat{v},r)$ and $Y:B\to A$ be defined by $Y(v',y',r')=(v',y'+r'\hat{v'},r')$. \\
    \underline{\textit{Step 1} }. $X:A\to B$ is well-defined. \\
    To verify well-definedness of $X$, for given $(v,y,r)\in A$, $y-r\hat{v} \in \O$ clearly. We also notice that 
    \begin{equation}
        0<r< r+|y- q_+(y,v)|=|(y-r\hat{v})-q_+(y-r\hat{v},v)|.
    \end{equation}
    Therefore, $X(v,y,r)\in B$ and $X$ is well-defined. \\
    \underline{\textit{Step 2} }. $Y:B\to A$ is well-defined. \\
    To verify well-definedness of $Y$, for given $(v',y',r')\in B$, $y'+r'\hat{v'} \in \O$ clearly. We note that 
    \begin{equation}
        0<r' < r'+|y'-q_-(y',v')|=|(y'+r'\hat{v'}) -q_-(y'+r'\hat{v'},v')|.
    \end{equation}
    Therefore, $Y(v',y',r')\in A$ and $Y$ is well-defined.  \\
    Since $X\circ Y =Y\circ X =\text{id}$ and the Jacobian for the change of variables is clearly constant $1$, we have 
    \eqref{eq:ChangeOfVariable1}.
\end{proof}
\begin{lem}  \label{lemma:ChangeOfVariable2}
We have the following formula for change of variables for any non-negative measurable function $h=h(v,y,x,r)$.
\begin{equation}\begin{split}  \label{eq:ChangeOfVariable2}
    &\int_{D_1}\int^{|y-q_-(y,v)|}_{|x-q_-(x,v)|} h(v,y,x,r) \, drdxdydv \\
    &= \int_{\R^3}\int_{\O}\int_{\O_{v',y'}}\int^{\min\{|x'-q^2(x',v')|,|y'-q_+(y',v')|\}}_{|x'-q^1(x',v')|}  \\
    &\qquad \qquad \qquad  h(v',y'+r'\hat{v'},x'+r'\hat{v'},r') \, dr'dx'dy'dv', 
\end{split}\end{equation}
    where $D_1$ is as defined by \eqref{eq:domainD1} and 
    $\O_{v',y'}$ is defined by (see Figure~\ref{fig:omegavy})
    \begin{equation} \label{eq:OmegavyDef}\begin{split}
        \O_{v',y'}=&\large\{ x'\notin \O : \, \text{there exist }t^i(x',v'),\, i=1,2, \\
        &\text{such that }0<t^1(x',v')<t^2(x',v'), \\
        & q^i(x',v'):=x'+t^i(x',v')\hat{v'}\in \partial\O, \\
        &t^1(x',v') < |y'-q_+(y',v')| \large\}.
    \end{split}\end{equation}
\end{lem}
\begin{figure} [!htb]
    \centering
    \includegraphics{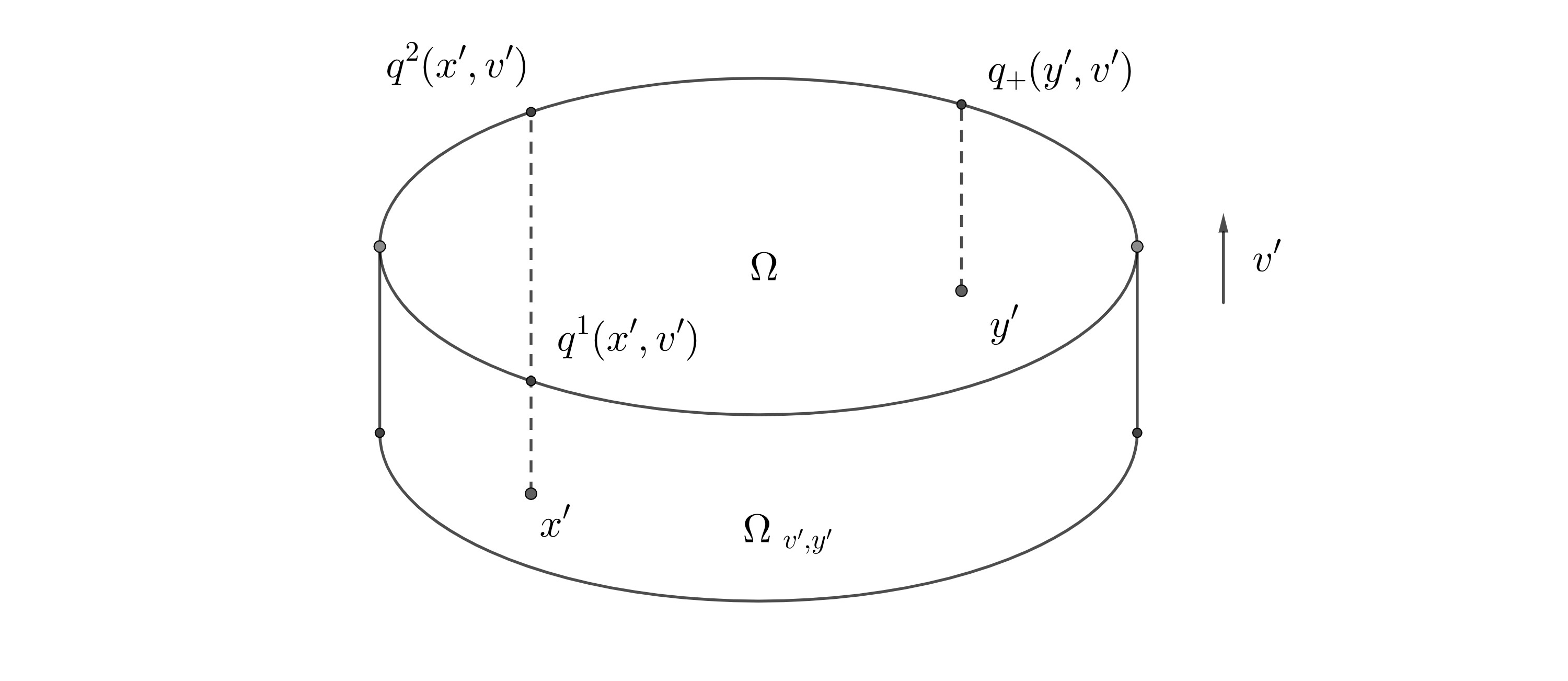}
    \caption{}
    \label{fig:omegavy}
\end{figure}
\begin{rmk}
    An alternative way to characterize $\O_{v',y'}$ is as follows. We first define
    \begin{equation}
        \O'_{v',y'}=\large\{ z-t \hat{v'}: \, z\in \O, \, 0<t<|y' - q_+(y',v')| \large\}.
    \end{equation}
    Then, $\O_{v',y'}$ can be expressed as 
    \begin{equation}
        \O_{v',y'} = \O'_{v',y'} \setminus \bar{\O}.
    \end{equation}
\end{rmk}
\begin{proof}
    We denote the domain of integration on the left hand side and right hand side by $\check{A}$ and $\check{B}$, 
    respectively. This time we consider 
    \begin{equation} 
        \left\{ 
            \begin{array}{l}
                v'=v, \\
                y'=y-r\hat{v}, \\
                x'=x-r\hat{v}, \\
                r'=r.  
            \end{array} 
            \right.
        \end{equation}
        Let $\check{X}:\check{A}\to \check{B}$ be defined by $\check{X}(v,y,x,r)=(v,y-r\hat{v},x-r\hat{v},r)$ and $\check{Y}:\check{B}\to \check{A}$ be defined by 
        $\check{Y}(v',y',x',r')=(v',y'+r'\hat{v'},x'+r'\hat{v'},r')$. \\ 
        \underline{\textit{Step 1} }. $\check{X}:\check{A}\to \check{B}$ is well-defined. \\
        For given $(v,y,x,r)\in \check{A}$, clearly we have 
        $y-r\hat{v}\in \O$ and $x-r\hat{v}\notin \O$ and there are two positive numbers 
        $t^1(x-r\hat{v},v)<t^2(x-r\hat{v},v)$ such that 
        \[
            q^i(x-r\hat{v},v)=x-r\hat{v}+t^i(x-r\hat{v},v)\hat{v}\in \partial\O.
            \]
        Moreover, comparing the distances results in 
        \begin{equation}\begin{split}
            t^1(x-r\hat{v},v)&=|(x-r\hat{v})-q^1(x-r\hat{v},v)| \\
            &< |(x-r\hat{v})-x| \\
            &= r\\
            &= |(y-r\hat{v})-y| \\
            &< |(y-r\hat{v})-q_+(y-r\hat{v},v)|. 
        \end{split}\end{equation}
        On the other hand, we notice
        \begin{equation}\begin{split}
            r&= |(x-r\hat{v})-x| \\
            &< |(x-r\hat{v})-q^2(x-r\hat{v},v)|.
        \end{split}\end{equation}
        We conlcude that $\check{X}(v,y,x,r)\in \check{B}$ and therefore $\check{X}$ is well-defined. \\
        \underline{\textit{Step 2} }. $\check{Y}:\check{B}\to \check{A}$ is well-defined. \\
        For given $(v',y',x',r')\in \check{B}$, since $r' \leq  |y' -q_+(y',v')|$, we have 
        $y'+r'\hat{v'}\in\O$. Likewise, since 
        \[
            |x'-q^1(x',v')|<r<|x'-q^2(x',v')|,    
        \]
        we have $x'+r'\hat{v'}\in\O$. 
        Comparing the distances yields 
        \begin{equation}\begin{split}
            |(x'+r'\hat{v'})-q_-(x'+r'\hat{v'},v')| &< |(x'+r'\hat{v'})-x'| \\
            &=r' \\
            &=|(y'+r'\hat{v'})-y'|  \\
            &<|(y'+r'\hat{v'}) - q_-(y'+r'\hat{v'},v')|.
        \end{split}\end{equation} 
        Therefore, it follows that $\check{Y}(v',y',x',r')\in \check{A}$ and therefore $\check{Y}$ is 
        well-defined. \\
        Since $\check{X}\circ \check{Y} =\check{Y}\circ \check{X} =\text{id}$ and the Jacobian for the change of variables is clearly constant $1$, we conclude 
        \eqref{eq:ChangeOfVariable2}. This completes the proof.
\end{proof}
\begin{prop} \label{prop:SKsquare}
    Let $h$ be a function belonging to $L^2(\O\times\R^3)$. Then, there exists a constant $C$ such that 
    \begin{align} 
        |\SO Kh(y,v)|^2 \leq &C\, \frac{1}{|v|}\int_{\R^3}\int^{|y-q_-(y,v)|}_0 |k(v,v_*)||h(y-r\hat{v},v_*)|^2 \,drdv_*, \label{eq:SKsquare} \\
        |K \SO h(y,v)|^2 \leq &C\, \int_{\R^3}\int^{|y-q_-(y,v_*)|}_0\frac{1}{|v_*|} |k(v,v_*)||h(y-r\hat{v_*},v_*)|^2 \,drdv_*. \label{eq:SKsquare2}
    \end{align}
\end{prop}
\begin{proof}
    By Cauchy-Schwarz inequality, we see that
    \begin{equation}\begin{split} \label{eq:SKsquareCS}
        |\SO K h(y,v)|^2 =&\left( \int^{\tau_-(y,v)}_0 e^{-\nu(v)s}Kh(y-sv,v)\,ds \right)^2 \\
        \leq& \left(\int^{\tau_-(y,v)}_0 e^{-2\nu(v)s}\,ds\right) \left( \int^{\tau_-(y,v)}_0 |Kh(y-sv,v)|^2\, ds\right)  \\
        \leq& C\, \int^{\tau_-(y,v)}_0 |Kh(y-sv,v)|^2\, ds  \\
        =&C\, \int^{|y-q_-(y,v)|}_0 \frac{1}{|v|} |Kh(y-r\hat{v},v)|^2\, dr, 
    \end{split}\end{equation}
    where $r=s|v|$. Furthermore, Corollary~\ref{cor:k} implies that 
    \begin{equation}\begin{split} \label{eq:SKsquareK}
        |Kh(y-r\hat{v},v)|^2\leq& \left(\int\limits_{\R^3} |k(v,v_*)|\,dv_*\right)\left(\int\limits_{\R^3} |k(v,v_*)||h(y-r\hat{v},v_*)|^2\,dv_*\right) \\
        \leq& C\, \int\limits_{\R^3} |k(v,v_*)||h(y-r\hat{v},v_*)|^2\,dv_*. 
    \end{split}\end{equation}
        Combining \eqref{eq:SKsquareCS} and \eqref{eq:SKsquareK} yields \eqref{eq:SKsquare}. 
        
        To prove \eqref{eq:SKsquare2}, we apply Cauchy-Schwarz inequality again to deduce that
        \begin{equation}\begin{split}
            |K\SO h(y,v)|^2 =& \left( \int_{\R^3}\int^{\tau_-(y,v_*)}_0 e^{-\nu(v_*)s} k(v,v_*) h(y-sv_*,v_*)\,dsdv_* \right)^2  \\
            \leq& \left(\int_{\R^3} |k(v,v_*)|\,dv_*\right) \left(\int_{\R^3}\int^{\tau_-(y,v_*)}_0 |k(v,v_*)||h(y-sv_*,v_*)|^2\,dsdv_*\right) \\
            \leq& C\, \int_{\R^3}\int^{\tau_-(y,v_*)}_0 |k(v,v_*)||h(y-sv_*,v_*)|^2\,dsdv_* \\
            =&C\, \int_{\R^3}\int^{|y-q_-(y,v_*)|}_0 \frac{1}{|v_*|} |k(v,v_*)||h(y-r\hat{v_*},v_*)|^2\,drdv_*,
        \end{split}\end{equation}
        where $r=s|v_*|$.
\end{proof}
We are now in a position to prove Lemma~\ref{lemma:soksok}.
\begin{proof}[Proof of Lemma~\ref{lemma:soksok}]
    For given $f\in L^2(\O\times\R^3)$, we recall $f_i$ denotes the function $(\SO K)^i f$. To prove the lemma, it suffices to show 
    \begin{equation}
        \int_{D_1}\frac{|f_2(x,v)-f_2(y,v)|^2}{|x-y|^4}\, dxdydv \leq C\, \| f\|^2_{L^2(\O\times\R^3)},
    \end{equation}
    where $D_1$ is as defined by \eqref{eq:domainD1}. According to Corollary~\ref{cor:ksok}, we have 
    \begin{equation} \label{eq:kf1}
    \| Kf_1 \|_{L^2_v(\R^3;H^{\sfrac 1 2}_x(\O))} \leq C\, \|f \|_{L^2(\O\times\R^3)}.
    \end{equation}
    In domain $D_1$, we notice that $\tau_-(x,v)\leq \tau_-(y,v)$. Therefore, 
    \begin{equation}\begin{split} \label{eq:somega2}
        &|S_\O Kf_1 (x,v)-S_\O Kf_1(y,v)|^2  \\
        &\qquad\leq 2 \left| \int^{\tau_-(x,v)}_0 e^{-\nu(v)s}(Kf_1(x-sv,v)-Kf_1(y-sv,v))\,ds \right|^2 \\
        &\qquad \quad+2 \left| \int^{\tau_-(y,v)}_{\tau_-(x,v)}e^{-\nu(v)s}Kf_1(y-sv,v)\,ds \right|^2. 
    \end{split}\end{equation}
    Hence, we have 
    \begin{equation}
        \int\limits_{D_1} \frac{|f_2(x,v)-f_2(y,v)|^2}{|x-y|^{4}}\, dxdydv \leq I_1 +I_2,
    \end{equation}
    where
    \begin{align}
        I_1 &:= \int_{D_1} \frac{2 \left| \int^{\tau_-(x,v)}_0 e^{-\nu(v)s}(Kf_1(x-sv,v)-Kf_1(y-sv,v))\,ds \right|^2}{|x-y|^{4}}\, dxdydv, \\
        I_2 &:= \int_{D_1} \frac{2 \left| \int^{\tau_-(y,v)}_{\tau_-(x,v)}e^{-\nu(v)s}Kf_1(y-sv,v)\,ds \right|^2}{|x-y|^{4}}\, dxdydv.
    \end{align}
    To estimate $I_1$, taking $\epsilon = \frac{1}{2}$ and $h=Kf_1$ in the steps of \eqref{eq:SOextension1} and 
    \eqref{eq:SOextension}, in the same fashion, we conclude
    \begin{equation} \label{eq:soksokI1}
        \begin{split}
        &\int\limits_{D_1}\frac{\left| \int^{\tau_-(x,v)}_0 e^{-\nu(v)s}(Kf_1(x-sv,v)-Kf_1(y-sv,v))\,ds \right|^2}{|x-y|^{4}}\,dxdydv  \\
        &\qquad\leq C\, \| Kf_1\|^2_{L^2_v(\R^3;H^{\sfrac{1}{2}}_x(\O))}  \\
        &\qquad \leq  C\, \| f\|^2_{L^2(\O\times\R^3)},
    \end{split}\end{equation}
    where we have used \eqref{eq:kf1}. 
    To deal with $I_2$, by Cauchy-Schwarz inequality again, we have
    \begin{equation} \label{eq:soksokEstimate1}
        \begin{split} 
        &\left| \int^{\tau_-(y,v)}_{\tau_-(x,v)}e^{-\nu(v)s}Kf_1(y-sv,v)\,ds \right|^2 \\
        &\qquad\leq\left(\int^{\tau_-(y,v)}_{\tau_-(x,v)}e^{-\nu(v)s} \,ds\right) \left(\int^{\tau_-(y,v)}_{\tau_-(x,v)}e^{-\nu(v)s}|Kf_1(y-sv,v)|^2\,ds \right).
    \end{split}\end{equation}
    For small $\epsilon>0$, 
    \begin{equation} \label{eq:soksokEstimate2}
        \begin{split}
        \int^{\tau_-(y,v)}_{\tau_-(x,v)}e^{-\nu(v)s} \,ds \leq&\left( \int^{\infty}_{0}e^{-\nu(v)s} \,ds\right)^\epsilon \left( \int^{\tau_-(y,v)}_{\tau_-(x,v)}e^{-\nu(v)s} \,ds\right)^{1-\epsilon} \\
        &\leq \frac{1}{\nu_0^\epsilon}|\tau_-(x,v)-\tau_-(y,v)|^{1-\epsilon} \\
        &\leq C\, \frac{|x-y|^{1-\epsilon}}{N_-(x,v)^{1-\epsilon}|v|^{1-\epsilon}},
    \end{split}\end{equation}
    where the last inequality follows from Proposition~\ref{prop:ProjDistance2}. On the other hand, 
    the change of variable $r=s|v|$ results in 
    \begin{equation} \label{eq:soksokEstimate3}
        \begin{split}
        \int^{\tau_-(y,v)}_{\tau_-(x,v)}e^{-\nu(v)s}|Kf_1(y-sv,v)|^2\,ds =& \frac{1}{|v|}\int^{|y-q_-(y,v)|}_{|x-q_-(x,v)|}e^{-\frac{\nu(v)}{|v|}r}|Kf_1(y-r\hat{v},v)|^2\,dr\\
        \leq&  \frac{1}{|v|}\int^{|y-q_-(y,v)|}_{|x-q_-(x,v)|}|Kf_1(y-r\hat{v},v)|^2\,dr.
        \end{split}
    \end{equation}
    Combining \eqref{eq:soksokEstimate1}-\eqref{eq:soksokEstimate3}, we have
    \begin{equation} \label{eq:soksokI2step1}
        \begin{split}
        &\left| \int^{\tau_-(y,v)}_{\tau_-(x,v)}e^{-\nu(v)s}Kf_1(y-sv,v)\,ds \right|^2 \\
        &\qquad\leq C\, \frac{|x-y|^{1-\epsilon}}{N_-(x,v)^{1-\epsilon}|v|^{2-\epsilon}}\int^{|y-q_-(y,v)|}_{|x-q_-(x,v)|}|Kf_1(y-r\hat{v},v)|^2\,dr. 
    \end{split}\end{equation}
    Therefore, it follows that 
    \begin{equation} \label{eq:soksokI2step2}
        I_2  \leq C\, \int_{D_1}\int^{|y-q_-(y,v)|}_{|x-q_-(x,v)|} \frac{|Kf_1(y-r\hat{v},v)|^2}{N_-(x,v)^{1-\epsilon}|v|^{2-\epsilon}|x-y|^{3+\epsilon}}\, dr dxdydv.
    \end{equation}
    Letting $x'=x-r\hat{v}$ and $y'=y-r\hat{v}$, by Lemma~\ref{lemma:ChangeOfVariable2}, we obtain
    \begin{equation} \label{eq:soksokI2step3}
        \begin{split}
        &\int_{D_1} \int^{|y-q_-(y,v)|}_{|x-q_-(x,v)|} \frac{|Kf_1(y-r\hat{v},v)|^2}{N_-(x,v)^{1-\epsilon}|v|^{2-\epsilon}|x-y|^{3+\epsilon}}\, dr dxdydv \\
        &\quad=\int\limits_{\R^3}\int\limits_{\O}\int\limits_{\O_{v,y'}}\int^{\min\{|q^2(x',v)-x'|,|q_+(y',v)-y'|\}}_{|q^1(x',v)-x'|} \\
        &\qquad \qquad \frac{|Kf_1(y',v)|^2}{N_-(q^1(x',v),v)^{1-\epsilon}|v|^{2-\epsilon}|x'-y'|^{3+\epsilon}}\, dr dx'dy'dv \\
        &\quad\leq \int\limits_{\R^3}\int\limits_{\O}\int\limits_{\O_{v,y'}}  \frac{|Kf_1(y',v)|^2 |q^2(x',v)-q^1(x',v)| }{N_-(q^1(x',v),v)^{1-\epsilon}|v|^{2-\epsilon}|x'-y'|^{3+\epsilon}}\,dx'dy'dv \\
        &\quad \leq C\, \int\limits_{\R^3}\int\limits_{\O}\int\limits_{\O_{v,y'}}  \frac{|Kf_1(y',v)|^2 }{|v|^{2-\epsilon}|x'-y'|^{3+\epsilon}}\,dx'dy'dv,
    \end{split}\end{equation}
    where we have utilized Proposition~\ref{prop:ChordBound} by 
    identifying $q^1(x',v)=q_-(x,v)$ and $q^2(x',v)=q_+(x,v)$. Noticing that from \eqref{eq:OmegavyDef}, 
    $\O_{v,y'}$ lies outside of $\O$, we have
    \[
        \O_{v,y'}\subset \left(\R^3\setminus\O\right) \subset \left(\R^3\setminus B_{d_{y'}}(y')\right).
        \]
    With this in mind, we deduce
    \begin{equation} \label{eq:soksokI2step4}
        \begin{split}
        I_2&\leq C\,\int\limits_{\R^3}\int\limits_{\O}\int\limits_{\O_{v,y'}}  \frac{|Kf_1(y',v)|^2 }{|v|^{2-\epsilon}|x'-y'|^{3+\epsilon}}\,dx'dy'dv \\
        &\leq \int\limits_{\R^3}\int\limits_{\O}\int_{\R^3\setminus B_{d_{y'}}(y')}  \frac{|Kf_1(y',v)|^2 }{|v|^{2-\epsilon}|x'-y'|^{3+\epsilon}}\,dx'dy'dv\\
        &\leq C\,\int\limits_{\R^3}\int\limits_{\O} \frac{|Kf_1(y',v)|^2 }{|v|^{2-\epsilon}d^{\epsilon}_{y'}}\,dy'dv. \\
    \end{split}\end{equation}
    Proposition~\ref{prop:SKsquare} and Lemma~\ref{lemma:ChangeOfVariable1} imply
    \begin{equation} \label{eq:soksokI2step5}
        \begin{split}
        &\int\limits_{\R^3}\int\limits_{\O} \frac{|Kf_1(y',v)|^2 }{|v|^{2-\epsilon}d^{\epsilon}_{y'}}\,dy'dv \\
        &\quad\leq C\, \int\limits_{\R^3}\int\limits_{\O}\int\limits_{\R^3} \int^{|y'-q_-(y',v_*)|}_0\frac{ |k(v,v_*)||Kf(y'-r\hat{v_*},v_*)|^2}{|v_*||v|^{2-\epsilon}d_{y'}^\epsilon}  \,drdv_*dy' dv \\
        &\quad= C\, \int\limits_{\R^3}\int\limits_{\O}\int\limits_{\R^3} \int^{|y''-q_+(y'',v_*)|}_0\frac{ |k(v,v_*)||Kf(y'',v_*)|^2}{|v_*||v|^{2-\epsilon}d_{y''+r\hat{v_*}}^\epsilon}  \,drdv_*dy'' dv \\
        &\quad\leq C\, \int\limits_{\R^3}\int\limits_{\O}\int\limits_{\R^3} \frac{ |k(v,v_*)||Kf(y'',v_*)|^2}{|v_*||v|^{2-\epsilon}}  \,dv_*dy'' dv,
    \end{split}\end{equation}
    where $y''=y'-r\hat{v_*}$ and we have used Lemma~\ref{lemma:FracIntegralOverChord} in the last line. 
    Utilizing Proposition~\ref{prop:1overvk}, we conclude 
    \begin{equation} \label{eq:soksokI2step6}
        \begin{split}
        I_2&\leq C\, \int\limits_{\O}\int\limits_{\R^3}\int\limits_{\R^3} \frac{ |k(v,v_*)||Kf(y'',v_*)|^2}{|v_*||v|^{2-\epsilon}}  \,dvdv_* dy'' \\
        &\leq C\,\int\limits_{\O}\int\limits_{\R^3}\frac{|Kf(y'',v_*)|^2}{|v_*|}\, dv_* dy'' \\
        &\leq C\, \int\limits_{\O} \int\limits_{\R^3} \frac{1}{|v_*|} \left(\int\limits_{\R^3}|k(v_*,w)|\,dw \right)\left(\int\limits_{\R^3}|k(v_*,w)||f(y'',w)|^2\,dw \right)\, dv_*dy'' \\
        &\leq C\, \int\limits_{\O} \int\limits_{\R^3}\int\limits_{\R^3} \frac{|k(v_*,w)||f(y'',w)|^2}{|v_*|}\,dv_* dw dy'' \\
        &\leq C\, \int\limits_{\O} \int\limits_{\R^3} |f(y'',w)|^2 dw dy''.
    \end{split}\end{equation}
    This completes the proof.
\end{proof}
Let us elaborate on Lemma~\ref{lemma:soksokextension} for a while. For 
$f\in L^2(\O\times \R^3)$, in view of Lemma~\ref{lemma:soksok}, the function 
$\SO K \SO Kf$ belongs to $L^2_v(\R^3;H^{\sfrac{1}{2}}_x(\O))$. Now consider its zero extension $\widetilde{\SO K \SO Kf}$ in the whole space $\R^3\times\R^3$. 
We claim 
that $\widetilde{\SO K \SO Kf}$ belongs to $L^2_v(\R^3;H^{\frac{1}{2}-\epsilon}_x(\R^3))$ for any $\epsilon\in(0,\frac 1 2)$ and 
\begin{equation}
    \| \widetilde{\SO K \SO Kf}\|^2_{L^2_v(\R^3;H^{\frac{1}{2}-\epsilon}_x(\R^3))} \leq \frac{C}{\epsilon}\, \| f \|^2_{ L^2(\O\times \R^3)}
\end{equation}
for some constant $C$.
\begin{proof}[Proof of Lemma~\ref{lemma:soksokextension}] 
    Recall $f_1:=\SO Kf$ and $f_2:=\SO K \SO Kf$.
    Since $\widetilde{f_2}(\cdot,v)$ vanishes outside of $\O$, we have  
    \begin{equation} \label{eq:soksokextension}
        \int\limits_{\R^3}\int\limits_{\R^3}\int\limits_{\R^3} \frac{|\widetilde{f_2}(x,v)-\widetilde{f_2}(y,v)|^2}{|x-y|^{4-2\epsilon}} \,dxdydv  = I_1 +  I_2 + I_3,
    \end{equation}
    where 
    \begin{align}
        I_1:=& \int\limits_{\R^3}\int\limits_{\O}\int\limits_{\O} \frac{|f_2(x,v)-f_2(y,v)|^2}{|x-y|^{4-2\epsilon}} \,dxdydv, \\
        I_2:=& \int\limits_{\R^3}\int\limits_{\O}\int\limits_{\R^3\setminus\O} \frac{|f_2(y,v)|^2}{|x-y|^{4-2\epsilon}} \,dxdydv, \\
        I_3:=& \int\limits_{\R^3}\int\limits_{\R^3\setminus\O}\int\limits_{\O} \frac{|f_2(x,v)|^2}{|x-y|^{4-2\epsilon}} \,dxdydv.
    \end{align}
    Regarding the boundedness of $I_1$, by Lemma~\ref{lemma:soksok}, we have 
    \begin{equation} \label{eq:I1soksokextension}
        \begin{split}
        I_1& =\int\limits_{\R^3}\int\limits_{\O}\int\limits_{\O} \frac{|f_2(x,v)-f_2(y,v)|^2}{|x-y|^{4}}\,|x-y|^{2\epsilon} \,dxdydv \\
        &\leq\max\{\diam\O,1\} \int\limits_{\R^3}\int\limits_{\O}\int\limits_{\O} \frac{|f_2(x,v)-f_2(y,v)|^2}{|x-y|^{4}}\,dxdydv \\
        &\leq C\, \| f \|^2_{L^2(\O\times\R^3)}.
    \end{split}\end{equation}
    By symmetry, one can see that $I_2=I_3$. Consequently, the remaining task is to deal with $I_2$. 
    We observe that, for $y\in\O$, 
    \begin{equation} \label{eq:soksokextensionEstimate1}
        \int\limits_{\R^3\setminus\O}\frac{1}{|x-y|^{4-2\epsilon}}\,dx \leq  \int\limits_{\R^3\setminus B_{d_y}(y)}\frac{1}{|x-y|^{4-2\epsilon}}\,dx \leq C\, d_y^{-1+2\epsilon}.
    \end{equation}
    By Proposition~\ref{prop:SKsquare}, we have 
    \begin{equation} \label{eq:soksokextensionEstimate2}
        \begin{split}
        |f_2(y,v)|^2=&|\SO K f_1(y,v)|^2 \\
        \leq& C\, \int_{\R^3}\int^{|y-q_-(y,v)|}_0  \frac{1}{|v|} |k(v,v_*)||f_1(y-r\hat{v},v_*)|^2\,dr dv_*.
    \end{split}\end{equation}
    Therefore, first combining \eqref{eq:soksokextensionEstimate1} and \eqref{eq:soksokextensionEstimate2} and performing 
    the change of variable 
    $y'=y-r\hat{v}$ as in Lemma~\ref{lemma:ChangeOfVariable1}, we deduce that 
    \begin{equation}\begin{split}
        I_2 &\leq  C \,\int\limits_{\R^3}\int\limits_{\O}\int\limits_{\R^3}\int^{|y-q_-(y,v)|}_0  \frac{1}{|v|}| k(v,v_*)||f_1(y-r\hat{v},v_*)|^2 d_{y}^{-1+2\epsilon} \,dr dy dv dv_* \\
        & = C\,\int\limits_{\R^3}\int\limits_{\O}\int\limits_{\R^3}\int^{|y'-q_+(y'v)|}_0  \frac{1}{|v|}| k(v,v_*)||f_1(y',v_*)|^2 d_{y'+r\hat{v}}^{-1+2\epsilon} \,dr dy' dv dv_* \\
        & \leq \frac{C}{\epsilon}\,\int\limits_{\R^3}\int\limits_{\O}\int\limits_{\R^3} \left( \frac{1}{|v|}| k(v,v_*)|\,dv \right) \,|f_1(y',v_*)|^2d_{y'}^{-\frac{1}{2}+2\epsilon} \, dy' dv_* \\
        & \leq \frac{C}{\epsilon}\, \int\limits_{\R^3}\int\limits_{\O} |f_1(y',v_*)|^2d_{y'}^{-\frac{1}{2}+2\epsilon}\,dy' dv_*,
    \end{split}\end{equation}
    where the third inequality follows from Lemma~\ref{lemma:FracIntegralOverChord} and the last inequality follows from Proposition~\ref{prop:1overvk}. 
    Applying Proposition~\ref{prop:SKsquare} again, we have 
    \begin{equation}
        |f_1(y',v_*)|^2\leq C\, \int_{\R^3}\int^{|y'-q_-(y',v_*)|}_0  \frac{1}{|v_*|} |k(v_*,w)||f(y'-r\hat{v_*},w)|^2\,drdw.
    \end{equation}
    Then, it follows that 
    \begin{equation}\begin{split} \label{eq:fdintegral}
        &\int\limits_{\R^3}\int\limits_{\O} |f_1(y',v_*)|^2d_{y'}^{-\frac{1}{2}+2\epsilon}\,dy' dv_*  \\
        &\quad \leq C\, \int\limits_{\R^3}\int\limits_{\O} \int\limits_{\R^3} \int^{|y'-q_-(y',v_*)|}_0 \frac{1}{|v_*|} |k(v_*,w)||f(y'-r\hat{v_*},w)|^2 d_{y'}^{-\frac{1}{2}+2\epsilon}\,dr dw dy' dv_*  \\
        &\quad = C\, \int\limits_{\R^3}\int\limits_{\R^3} \int\limits_{\O} \int^{|y''-q_+(y'',v_*)|}_0 \frac{1}{|v_*|} |k(v_*,w)||f(y'',w)|^2 d_{y''+r\hat{v_*}}^{-\frac{1}{2}+2\epsilon}\,dr dy'' dv_* dw, 
    \end{split}\end{equation} \label{eq:fdintegral2}
    \noindent where $y''=y'-r\hat{v_*}$. Similarly, by Lemma~\ref{lemma:FracIntegralOverChord} and Proposition~\ref{prop:1overvk}, we conclude that 
    \begin{equation} \label{eq:I2soksokextension}
        \begin{split}
        I_2 & \leq \frac{C}{\epsilon}\, \int\limits_{\R^3}\int\limits_{\O} |f_1(y',v_*)|^2d_{y'}^{-\frac{1}{2}+2\epsilon}\,dy' dv_* \\
        & \leq \frac{C}{\epsilon}\, \int\limits_{\R^3}\int\limits_{\O} \int\limits_{\R^3}  \frac{1}{|v_*|} |k(v_*,w)||f(y'',w)|^2 \, dv_* dy'' dw \\
        & \leq \frac{C}{\epsilon}\, \int\limits_{\R^3}\int\limits_{\O} |f(y'',w)|^2\, dy'' dw.
    \end{split}\end{equation}
    Combining \eqref{eq:I1soksokextension} and \eqref{eq:I2soksokextension} completes the proof.
\end{proof}
Lemma~\ref{lemma:soksokextension} allows us to improve regularity results via Lemma~\ref{lemma:ksk}.
\begin{cor} \label{cor:ksoksoksok}
    The operator $K\SO K\SO K\SO K :L^2(\O\times\R^3)\to L^2_v(\R^3;H^{1-\epsilon}_x(\O))$ is bounded for any $\epsilon\in (0,\frac 1 2)$.
\end{cor}
\begin{proof}
    For given $f\in L^2(\O\times\R^3)$ and $\epsilon\in (0,\frac 1 2)$, Lemma~\ref{lemma:soksokextension} implies 
    \begin{equation}
    \widetilde{\SO K \SO Kf} \in L^2_v(\R^3;H^{\frac{1}{2}-\epsilon}_x(\R^3)) = L^2_v(\R^3;\tilde{H}^{\frac{1}{2}-\epsilon}_x(\R^3)).
    \end{equation} 
    By Lemma~\ref{lemma:ksk}, we obtain 
    \begin{equation}
        KSK(\widetilde{\SO K \SO Kf})\in L^2_v(\R^3;\tilde{H}^{1-\epsilon}_x(\R^3)).
    \end{equation}
    Since
    \begin{equation}
        \left.\biggl(KSK(\widetilde{\SO K \SO Kf})\biggr) \right|_{\O\times \R^3} = K\SO K\SO K\SO K f,
    \end{equation}
    consequently $K\SO K\SO K\SO K f$ belongs to $L^2_v(\R^3;H^{1-\epsilon}_x(\O))$. The boundedness of the operator is due to 
    boundedness of $KSK$ and $Z\SO K \SO K$.
\end{proof}
We are now ready to prove the last lemma in this section.
\begin{lem}
    The operator $\SO K\SO K\SO K\SO K :L^2(\O\times\R^3)\to L^2_v(\R^3;H^{1-\epsilon}_x(\O))$ is bounded for any $\epsilon\in (0,\frac 1 2)$. Therefore, $f_4\in L^2_v(\R^3;H^{1-\epsilon}_x(\O))$.
\end{lem}
\begin{proof}
    We shall prove this lemma in a similar fashion as we prove Lemma~\ref{lemma:soksok}. 
    To do so, for given $f\in L^2(\O\times\R^3)$ and $\epsilon\in(0,\frac 1 2)$, it suffices to show 
    \begin{equation}
        \int_{D_1}\frac{|f_4(x,v)-f_4(y,v)|^2}{|x-y|^{5-2\epsilon}}\, dxdydv \leq C\, \| f\|^2_{L^2(\O\times\R^3)},
    \end{equation}
    where $D_1$ is as defined by \eqref{eq:domainD1}. According to Corollary~\ref{cor:ksoksoksok}, it follows that  
    \begin{equation} \label{eq:kf3}
        \| Kf_3 \|_{ L^2_v(\R^3;H^{1-\epsilon}_x(\O))} \leq C\, \| f\|_{L^2(\O\times\R^3)}.
    \end{equation}
    In domain $D_1$, we have 
    \begin{equation}\begin{split}
        &|S_\O Kf_3 (x,v)-S_\O Kf_3(y,v)|^2 \\
        &\qquad \leq 2 \left| \int^{\tau_-(x,v)}_0 e^{-\nu(v)s}(Kf_3(x-sv,v)-Kf_3(y-sv,v))\,ds \right|^2 \\
        &\qquad \quad+2 \left| \int^{\tau_-(y,v)}_{\tau_-(x,v)}e^{-\nu(v)s}Kf_3(y-sv,v)\,ds \right|^2. 
    \end{split}\end{equation}
    Therefore, we have 
    \begin{equation}
        \int\limits_{D_1} \frac{|f_4(x,v)-f_4(y,v)|^2}{|x-y|^{5-2\epsilon}}\, dxdydv \leq I_1 +I_2,
    \end{equation}
    where
    \begin{align}
        I_1 &:= \int_{D_1} \frac{2 \left| \int^{\tau_-(x,v)}_0 e^{-\nu(v)s}(Kf_3(x-sv,v)-Kf_3(y-sv,v))\,ds \right|^2}{|x-y|^{5-2\epsilon}}\, dxdydv, \\
        I_2 &:= \int_{D_1} \frac{2 \left| \int^{\tau_-(y,v)}_{\tau_-(x,v)}e^{-\nu(v)s}Kf_3(y-sv,v)\,ds \right|^2}{|x-y|^{5-2\epsilon}}\, dxdydv.
    \end{align}
    To estimate $I_1$, taking $h=Kf_3$ in the steps of \eqref{eq:SOextension1} and 
    \eqref{eq:SOextension}, in the same fashion, we conclude
    \begin{equation}
        \begin{split}
        &\int\limits_{D_1}\frac{\left| \int^{\tau_-(x,v)}_0 e^{-\nu(v)s}(Kf_3(x-sv,v)-Kf_3(y-sv,v))\,ds \right|^2}{|x-y|^{5-2\epsilon}}\,dxdydv  \\
        &\qquad\leq C\, \| Kf_3\|^2_{L^2_v(\R^3;H^{1-\epsilon}_x(\O))}  \\
        &\qquad \leq  C\, \| f\|^2_{L^2(\O\times\R^3)},
        \end{split}
    \end{equation}
    where we have used \eqref{eq:kf3}. 
    Concerning $I_2$, we proceed as in \eqref{eq:soksokI2step1}-\eqref{eq:soksokI2step5} to deduce 
    \begin{equation}\begin{split}
        I_2& \leq C\, \int\limits_{\R^3}\int\limits_{\O}\int\limits_{\O}\int^{|y-q_-(y,v)|}_{|x-q_-(x,v)|} \frac{|Kf_3(y-r\hat{v},v)|^2}{N_-(x,v)^{1-\epsilon}|v|^{2-\epsilon}|x-y|^{4-\epsilon}}\, dr dxdydv \\
        & =C\,\int\limits_{\R^3}\int\limits_{\O}\int\limits_{\O_{v,y'}}\int^{\min\{|q^2(x',v)-x'|,|q_+(y',v)-y'|\}}_{|q^1(x',v)-x'|}\\
        &\qquad\quad \frac{|Kf_3(y',v)|^2}{N_-(q^1(x',v),v)^{1-\epsilon}|v|^{2-\epsilon}|x'-y'|^{4-\epsilon}}\, dr dx'dy'dv \\
        & \leq C\, \int\limits_{\R^3}\int\limits_{\O}\int\limits_{\O_{v,y'}}  \frac{|Kf_3(y',v)|^2 }{|v|^{2-\epsilon}|x'-y'|^{4-\epsilon}}\,dx'dy'dv 
    \end{split}\end{equation}
    \begin{equation}\begin{split}
        \quad& \leq C\,\int\limits_{\R^3}\int\limits_{\O} \frac{|Kf_3(y',v)|^2 }{|v|^{2-\epsilon}d^{1-\epsilon}_{y'}}\,dy'dv \\
        \quad& \leq C\, \int\limits_{\R^3}\int\limits_{\O}\int\limits_{\R^3} \int^{|y'-q_-(y',v_*)|}_0\frac{ |k(v,v_*)||Kf_2(y'-r\hat{v_*},v_*)|^2}{|v_*||v|^{2-\epsilon}d_{y'}^{1-\epsilon}}  \,drdv_*dy' dv \\
        \quad&  =C\, \int\limits_{\R^3}\int\limits_{\O}\int\limits_{\R^3} \int^{|y''-q_+(y'',v_*)|}_0\frac{ |k(v,v_*)||Kf_2(y'',v_*)|^2}{|v_*||v|^{2-\epsilon}d_{y''+r\hat{v_*}}^{1-\epsilon}}  \,drdv_*dy'' dv \\
        \quad& \leq C\, \int\limits_{\R^3}\int\limits_{\O}\int\limits_{\R^3} \frac{ |k(v,v_*)||Kf_2(y'',v_*)|^2}{|v_*||v|^{2-\epsilon}d_{y''}^{\frac{1}{2}-\epsilon}}  \,dv_*dy'' dv,
    \end{split}\end{equation}
    where we utilized \eqref{eq:FracIntegralOverChord2} of Lemma~\ref{lemma:FracIntegralOverChord} in the last line instead. 
    Continuing as in \eqref{eq:soksokI2step6} yields 
    \begin{equation}\begin{split}
        I_2&\leq C\, \int\limits_{\R^3}\int\limits_{\O}\int\limits_{\R^3} \frac{ |k(v,v_*)||Kf_2(y'',v_*)|^2}{|v_*||v|^{2-\epsilon}d_{y''}^{\frac{1}{2}-\epsilon}}  \,dv dy'' dv_* \\
        &\leq C\,\int\limits_{\O}\int\limits_{\R^3}\frac{|Kf_2(y'',v_*)|^2}{|v_*|d_{y''}^{\frac{1}{2}-\epsilon}}\, dv_* dy'' \\
        &\leq C\, \int\limits_{\O} \int\limits_{\R^3} \frac{1}{|v_*|d_{y''}^{\frac{1}{2}-\epsilon}} \left(\int\limits_{\R^3}|k(v_*,w)|\,dw \right)\left(\int\limits_{\R^3}|k(v_*,w)||f_2(y'',w)|^2\,dw \right)\, dv_*dy'' \\
        &\leq C\, \int\limits_{\O} \int\limits_{\R^3}\int\limits_{\R^3} \frac{|k(v_*,w)||f_2(y'',w)|^2}{|v_*|d_{y''}^{\frac{1}{2}-\epsilon}}\,dv_* dw dy'' \\
        &\leq C\, \int\limits_{\O} \int\limits_{\R^3} \frac{|f_2(y'',w)|^2}{d_{y''}^{\frac{1}{2}-\epsilon}} dw dy''.
    \end{split}\end{equation}
    The last line above is \eqref{eq:fdintegral} with $f_2$ in place of $f_1$. In the same fashion, we can conclude that 
    \begin{equation}\begin{split}
        I_2 \leq& C\, \int\limits_{\O} \int\limits_{\R^3} \frac{|f_2(y'',w)|^2}{d_{y''}^{\frac{1}{2}-\epsilon}} dw dy'' \\
        \leq& C\, \int\limits_{\R^3}\int\limits_{\O} |f_1(z,w_*)|^2\, dz dw_* \\
        \leq& C\, \| f \|^2_{L^2(\O\times\R^3)}. \\
    \end{split}\end{equation}
    This completes the proof.
\end{proof}
\begin{rmk}
From our calculation, we are not able to show that $\widetilde{f_4}\in L^2_v(\R^3;H^{\frac{1}{2}}_x(\R^3))$. As a result, we 
can not further improve regularity by the same method as in Corollary~\ref{cor:ksoksoksok}.
\end{rmk}
~\\
~\\
\noindent \textbf{Acknowledgements: }The first author is supported in part by MOST grant 108-2628-M-002 -006 -MY4 and 106-2115-M-002 -011 -MY2. 
The third author is supported  by NCTS and MOST grant 104-2628-M-002-007-MY3.

    \end{document}